\documentclass[12pt,a4paper]{amsart}
\usepackage{times}
\usepackage{amsfonts}
\usepackage{amsthm}
\usepackage{amsmath}
\usepackage{amscd}
\usepackage{amssymb}
\usepackage[latin2]{inputenc}
\usepackage{t1enc}
\usepackage[mathscr]{eucal}
\usepackage{indentfirst}
\usepackage{graphicx}
\usepackage{graphics}
\usepackage{pict2e}
\usepackage{epic}
\numberwithin{equation}{section}
\usepackage[margin=2.9cm]{geometry}
\usepackage{epstopdf} 
\usepackage{mathtools}
\usepackage{xcolor}
\usepackage{amsrefs}
\usepackage{hyperref}
\usepackage{todonotes}

\makeatletter
\newcommand*{\rom}[1]{\expandafter\@slowromancap\romannumeral #1@}
\makeatother

\theoremstyle{plain}
\newtheorem{theorem}{Theorem}[section]
\newtheorem{lemma}[theorem]{Lemma}

\newtheorem{cor}[theorem]{Corollary}

\newtheorem{definition}[theorem]{Definition}

\definecolor{darkred}{rgb}{1, 0.1, 0.3}

\newcommand{\myparagraph}[1]	{{\vspace*{0.08in}\noindent{\bf #1~}}}
\newcommand {\mm}[1] {\ifmmode{#1}\else{\mbox{\(#1\)}}\fi}
\newcommand{\denselist}{\itemsep 0pt\parsep=1pt\partopsep 0pt}

\newcommand{\myER}			{{Erd\H{o}s--R\'enyi}}

\newcommand{\myprob}			{\mathbb{P}} %{\mathrm{P}}

\newcommand{\myProb}				{\mathbb{P}}

\newcommand{\R}				{\mathbb{R}}

\newcommand{\myWSP}		{{well-separated clique-partitions family}}
\newcommand{\cliqueP}		{{clique-partition}}
\newcommand{\Bconst}		{{Besicovitch constant}}
\newcommand{\myBC}		{{\beta}}

\newcommand{\aC}			{{\mathrm{C}}}
\newcommand{\aK}			{{\mathsf{K}}}
\newcommand{\myE}			{{\mathbb{E}}}
\newcommand{\aF}		{{\mathrm{F}}}
\newcommand{\aI}		{{\mathrm{I}}}
\newcommand{\aY}		{{\mathrm{Y}}}

\newcommand{\origin}    {\mathbf{o}}
\newcommand{\pp}    {{q}}
\newcommand{\qqq}    {{p}}

\newcommand\norm[1]{\left\lVert#1\right\rVert}

\newcommand{\mysetting}  {{the standard-setting-R}}

\DeclarePairedDelimiter\ceil{\lceil}{\rceil}
\DeclarePairedDelimiter\floor{\lfloor}{\rfloor}

\DeclareMathOperator*{\argmax}{arg\,max}

\begin{document}

\title{On the clique number of noisy random geometric graphs}

\author{Matthew Kahle}
\address{Department of Mathematics, The Ohio State University, USA}
\email{mkahle@math.osu.edu}

\author{Minghao Tian}
\address{Computer Science and Engineering Dept., The Ohio State University, USA}
\email{tian.394@osu.edu}

\author{Yusu Wang}
\address{Hal{\i}c{\i}o\u{g}lu Data Science Institute, University of California San Diego, San Diego, CA, USA}
\email{yusuwang@ucsd.edu}

\date{\today}
\thanks{This work is partially supported by National Science Foundation grants DMS-1547357, DMS-1352386, CCF-1740761, and RI-1815697.}

%\affil[1]{Department of Mathematics, The Ohio State University, USA}
%\texttt{mkahle@math.osu.edu}}
%\affil[2]{Computer Science and Engineering Dept., The Ohio State University, USA}
%\texttt{tian.394@osu.edu, yusu@cse.ohio-state.edu}}

%\yusu{(1) why is abstract at the end? (2) I moved ``Acknowledgement" to the end of the paper. }

%\minghao{I followed the author's guideline of Random Structures \& Algorithms. It says that every submitted paper should be arranged in such way. Check section 3 in: \url{https://onlinelibrary.wiley.com/page/journal/10982418/homepage/forauthors.html}. Anyway, I just moved the abstract back.}

%\subjclass[2010]{05C80}

%\keywords{clique number, random geometric graphs, well-separated clique-partitions family, generalised scan statistic}

\begin{abstract}
Let $G_n$ be a random geometric graph, and then for $\pp,\qqq \in [0,1)$ we construct a \emph{$(\pp,\qqq)$-perturbed noisy random geometric graph} $G_n^{\pp,\qqq}$ where each existing edge in $G_n$ is removed with probability $\pp$, while and each non-existent edge in $G_n$ is inserted with probability $\qqq$.  We give asymptotically tight bounds on the clique number $\omega\left(G_n^{\pp,\qqq}\right)$ for several regimes of parameter.
\end{abstract}

\maketitle

\section{Introduction and statement of results}

The \emph{random geometric graph} $G(\mathcal{X}_n; r) = G_{\mathbb{R}^d}(X_1, X_2, \cdots, X_n; r)$ has vertices $\mathcal{X}_n = \{X_1, X_2, \cdots, X_n\}$ where the $X_i$ are $d$-dimensional variables sampled from a common probability distribution $\nu$ on $\mathbb{R}^d$, and an edge $(X_i, X_j)$ is added whenever $X_i$ and $X_j$ are within Euclidean distance $r = r(n) > 0$ to each other. Often, only mild assumptions on the underlying probability distribution on $\mathbb{R}^d$ are required, for example a bounded, measurable, density function. See Penrose's monograph \cite{penrose2003random} for an overview of the subject.

Random geometric graphs are  useful in  applications, e.g.\ wireless networks, transportation networks \cites{nekovee2007worm,blanchard2008mathematical}), etc.
We are mainly interested here in adding noise to a random geometric graph, in the sense of randomly adding either long-range edges, deleting short-range edges, or both.
Such graphs also arise naturally in applications, e.g., in modeling biological epidemics and collective social processes \cite{taylor2015topological}.
See \cite{mahler2018analysis} for work on contagion dynamics on noisy geometric networks. 

%Random geometric graphs have formed a central fields in random graph theory for some decades, partly because of the challenging open problems and partly because of its wide applications in real world where the physical locations of objects involved play an important role (e.g., wireless networks, transportation networks \cites{nekovee2007worm,blanchard2008mathematical}). 

The \emph{clique number} of a graph $G$, denoted $\omega(G)$, is the order of the largest complete subgraph in $G$. Clique numbers of various models of random graphs are well studied. For example, it is a well-known result of Matula \cite{matula1976largest} and independently Bollob\'as and Erd\H{o}s \cite{BE1976} that the clique number of the Erd\H{o}s--R\'enyi random graph $G(n,p)$ is concentrated on at most two values.

Penrose \cite{penrose2003random}, M{\"u}ller \cite{muller2008two} and McDiarmid \cite{mcdiarmid2011chromatic} , among others, studied the clique number $\omega\left(G\left(\mathcal{X}_n; r\right)\right)$ of random geometric graphs $G(\mathcal{X}_n; r)$ for different ranges of $r$. For example, M{\"u}ller showed that the clique number $\omega\left(G(\mathcal{X}_n; r)\right)$ satisfies a similar two-point concentration if $nr^d = o(\log n)$ \cite{muller2008two}. In \cites{mcdiarmid2011chromatic}, McDiarmid and M\"uller compared clique number and chromatic number for random geometric graphs, and showed that there is a sharp threshold between these quantities being nearly equal and the chromatic number being much larger. They use methods which inspire some of our work here. The case of random geometric graphs when the dimension $d$ tends to infinity with the number of vertices $n$ has also been studied, for example, in \cite{devroye2011high}, and for $d$ growing quickly it is known that they are close in total variation distance to the Erd\H{o}s--R\'enyi model

Here we consider the clique number for noisy random geometric graphs, studied earlier in \cite{ParthasarathyST17} where they were called ``ER-perturbed random geometric graphs''. More precisely, we call $G^{\pp,\qqq}(\mathcal{X}_n; r)$ a \emph{$(\pp,\qqq)$-perturbed noisy random geometric graph} if a \emph{$(\pp,\qqq)$-perturbation} is added to a random geometric graph $G(\mathcal{X}_n; r)$; that is, each edge in $G(\mathcal{X}_n; r)$ is deleted with a uniform probability $\pp$, while each possible edge between two unconnected nodes $u, v$ is inserted independently, with a uniform probability $\qqq$. % 
See Chapter 6 in \cite{penrose2003random} for a comprehensive overview of clique numbers of random geometric graphs, including laws of large numbers. For random geometric graphs sampled from well-behaved probability distributions on metric spaces (e.g., from the uniform distribution on a cube $[0,1]^d$), the clique number and the maximum degree are of the same order of magnitude. We will see that this does not generally hold for noisy random geometric graphs.
 
\subsection{Some definitions and notation}\label{subsec:definitions_notations}
Before we state our main results, we first state some more precise definitions.

\begin{definition}[Random geometric graph \cite{mcdiarmid2011chromatic}]\label{def:RGG}
Given a sequence of independent random points $X_1, X_2, \cdots$ in $\R^d$ sampled from a common probability distribution $\nu$ with bounded density function $f$ (that is, for any Borel set $A \subseteq \mathbb{R}^d$, $\nu(A) = \int_A f(x) dx$), and a positive distance $r=r(n)>0$, we construct a random geometric graph $G(\mathcal{X}_n;r)$ with vertex set $\mathcal{X}_n = \{X_1, \cdots, X_n\}$, where distinct $X_i$ and $X_j$ are adjacent when $\norm{X_i - X_j}\leq r$. Here $\norm{\cdot} $ may be any norm on $\R^d$. 
%We also use notation $G(\mathcal{X}_n;r) = (V, E)$, where each $v_i$ in the vertex set $V = \{v_1, v_2, \cdots, v_n\}$ corresponds to the random point $X_i$, while an edge $(v_i, v_j)\in E$ if and only if $\norm{X_i - X_j} \le r$.
%We also use notation $G(\mathcal{X}_n;r) = (V, E)$ where $V = \{v_1, v_2, \cdots, v_n\}$ are the set of vertices which coincides with $\mathcal{X}$ and $E$ is the set of corresponding edges. 
%\minghao{Not sure whether it is fine or not. Here, capital $X_i$'s indicate that those are r.v. (standard notation from Penrose). But later, when we talk about the graph self, maybe using small letters like $u,v$ to refer the vertices seems more clear. Anyway, currently I use these two systems of notations for noisy-RGG in the paper.}
\end{definition}

Denote $\sigma$ as the essential supremum of the probability density function $f$ of $\nu$, that is
\begin{align*}
\sigma := \sup \left\{ t: \int_ {\left\{ y: f(y) > t \right\}} dx >0 \right\}.
\end{align*}
We call $\sigma$ the maximum density of $\nu$. Denote $B_s(x) := \{y \in \R^d: \norm{y-x} \leq s\}$ as the ball centered at $x \in \R^d$ of radius $s$. Also set $\theta = \int_{B_{1}(\origin)} dx$, where $\origin$ is the origin of $\R^d$; that is, $\theta$ is the volume of any radius-$1$ ball in $\mathbb{R}^d$. 
%\yusu{Where is $B(\cdot; \cdot)$ defined? Also, $B(0; 1)$ looks really weird. Also, earlier, you have $B_{r/2}$ -- remember to use consistent notations in a paper. } \minghao{done.}

We now introduce our \emph{$(\pp,\qqq)$-perturbed noisy random geometric graphs} $G^{\pp,\qqq}(\mathcal{X}_n; r)$.
\begin{definition}[$(\pp,\qqq)$-perturbed noisy random geometric graph]\label{def:pq_perturbed_RGG}
Given a random geometric graph $G(\mathcal{X}_n;r)$ as in Definition \ref{def:RGG}, the $(\pp,\qqq)$-perturbed noisy random geometric graph $G^{\pp,\qqq}(\mathcal{X}_n;r)$ is obtained by deleting each existing edge in $G(\mathcal{X}_n;r)$ independently with probability $\pp$ as well as inserting each non-existent edge in $G(\mathcal{X}_n;r)$ independently with probability $\qqq$. 
\end{definition}

Note that the order of applying the above two types of perturbations doesn't matter since they are applied to two disjoint sets respectively. This process can be applied to any graph, and we call it a \emph{($\pp,\qqq$)-perturbation}. The resulting graph $G^{\pp,\qqq}(\mathcal{X}_n;r)$ is called a \emph{($\pp,\qqq$)-perturbation of $G(\mathcal{X}_n;r)$}, or simply a \emph{noisy random geometric graph}. 

%Throughout this paper, we use the standard Bachmann-Landau notation (asymptotic notation). That is, for real valued functions $f(n)$ and $g(n)$, as $n\rightarrow \infty$, we say
%\yusu{Minghao, I have similar comments on the choice of $K$ -- usually it is either a set, or an integer, not a constant. Also, if it is constant, you should say explicitly below. } \minghao{fixed.}
%\yusu{ Minghao: can we change constant $K$ to sth. else? Weird to have $K$ as a constant. Can we just use $c$?} \minghao{fixed.}
%\begin{enumerate}
%\item $f(n) = O(g(n))$: $\exists$ two constants $c > 0$ and $n_0 \in \mathbb{N}$ such that $|f(n)| \leq cg(n)$ for all $n \geq n_0$;
%\item $f(n) = o(g(n))$: $\forall \epsilon > 0$, $\exists n_0 \in \mathbb{N}$ such that $|f(n)| < \epsilon g(n)$ for all $n \geq n_0$;
%\item $f(n) = \Omega(g(n))$: $\exists$ two constants $c > 0 $ and $n_0 \in \mathbb{N}$ such that $f(n) \geq c g(n)$ for all $n \geq n_0$;
%\item $f(n) = \Theta(g(n))$: $f(n) = O(g(n))$ and $f(n) = \Omega(g(n))$;
%\end{enumerate}

Throughout this paper, we use the standard Bachmann-Landau notation (asymptotic notation). That is, for real valued functions $f(n)$ and $g(n)$, as $n\rightarrow \infty$, we say
\begin{enumerate}
\item $f(n) = O(g(n))$ if there exist constants $c > 0$ and $n_0 \in \mathbb{N}$ such that $|f(n)| \leq cg(n)$ for all $n \geq n_0$;
\item $f(n) = o(g(n))$ if for every  $\epsilon > 0$, there exists $n_0 \in \mathbb{N}$ such that $|f(n)| < \epsilon g(n)$ for all $n \geq n_0$;
\item $f(n) = \Omega(g(n))$: if there exists constants $c > 0 $ and $n_0 \in \mathbb{N}$ such that $f(n) \geq c g(n)$ for all $n \geq n_0$;
\item $f(n) = \Theta(g(n))$ if $f(n) = O(g(n))$ and $f(n) = \Omega(g(n))$;
\end{enumerate}

We also use the notation $f \ll g$ to mean that $f(n)/g(n) \to 0$ as $n \to \infty$, $f \lesssim g$ to mean that there exists a constant $C > 0$ such that $f(n)/g(n) < C$ for all sufficiently large $n$, $f \gtrsim g$ to mean that there exists a constant $c > 0$ such that $f(n)/g(n) > c$ for all sufficiently large $n$, and $f \sim g$ to mean that $f \lesssim g$ and $f \gtrsim g$.

Recall that a \emph{clique} in any graph $G$ is a set of vertices which are pairwise connected. In this paper, we use the standard notation $\omega(G)$ in graph theory to denote the \emph{clique number} of $G$, which is the largest cardinality of a clique in $G$.

%From the classic results in random geometric graph theory, many 
Many properties of $G(\mathcal{X}_n;r)$ are qualitatively different depending on which distance $r=r(n)$ is chosen. In some sense, the distance $r$ here plays a role similar to the edge-inserting probability $p(n)$ in \myER{} random graphs $G(n, p)$. 
We consider the following three regimes for the quantity $nr^d$: 
\begin{enumerate}\denselist
\item[\rom{1}.] (\emph{subcritical}) $nr^{d} \leq n^{-\alpha}$ for some fixed $\alpha > 0$;
\item[\rom{2}.] (\emph{``critical or nearly critical'' or ``thermodynamic''}) $n^{-\epsilon} \ll nr^d \ll \log {n}$ for all $\epsilon > 0$;
\item[\rom{3}.] (\emph{``supercritical''}) $\sigma n r^d / \log n \rightarrow t \in (0, \infty)$;
\end{enumerate}
In continuum percolation it is more standard to reserve critical for the case $nr^d \to \lambda_c$ for some special constant $\lambda_c >0$, so our terminology may be slightly nonstandard.

We often use the terminology \emph{almost surely} (or a.s.): In particular, if $\xi_1, \xi_2, \cdots$ is a sequence of random variables and $k_1,k_2,\cdots$ is a sequence of positive numbers, then ``\emph{a.s. $\xi_n \geq k_n$}'' means that $\lim_{n \rightarrow \infty} \myprob[\xi_n \geq k_n] = 1$. The other direction \emph{a.s. $\xi_n \leq k_n$} is defined similarly. 
Moreover, \emph{a.s. $\xi_n \lesssim k_n$} means that there exist $C_1 > 0$ such that $\lim_{n \rightarrow \infty} \myprob[\xi_n \leq C_1 k_n] = 1$. Similarly, we define a.s. $\xi_n \gtrsim k_n$ and a.s. $\xi_n \sim k_n$. We also use the terminology \emph{with high probability} (or w.h.p.): Specifically, if $A_1, A_2, \cdots$ is a sequence of events, then ``$A_n$ happens with high probability'' means that $\lim_{n \rightarrow \infty} \myprob[A_n] = 1$.

\paragraph{\bf Assumptions and notations for the remainder of the paper.~}
In what follows, unless specified explicitly, we assume the following setting throughout, which we refer to as \emph{\mysetting}: 
\begin{itemize}\denselist
\item The space we consider is the $d$-dimensional Euclidean space $\R^d$ with a fixed dimension $d$, equipped with some arbitrary norm $\norm{\cdot}$ on $\R^d$. 

\item $\theta = \int_{B_1(\origin)}dx$ is the volume of the unit ball $B_1(\origin) = \{x \in \R^d: \norm{x} \leq 1\}$.

\item $\myBC$ is the so-called \Bconst{} of $(\mathbb{R}^d, \norm{\cdot})$ (see Section \ref{sec:well_separated_family}). 
\item $\nu$ is a probability distribution with finite maximum density $\sigma$; and 
$X_1, X_2, \cdots$ are independent random variables sampled from $\nu$.
\item $r = \left(r(1), r(2), \cdots\right)$ is a sequence of positive real numbers such that $r(n) \rightarrow 0$ as $n \rightarrow \infty$.
\item $\pp$ and $\qqq=\qqq(n)$ are real numbers between 0 and 1 (for simplicity, we only consider the case when $\pp$ is a fixed constant). %\yusu{So $p$ is constant, but $q$ may not be?} \minghao{Yes.} \yusu{Minghao: then you need to say it explicitly. In particular, remember later you have a place with $p\in (0,1)$ where $p$ may not be constant. This would cause confusion for reader. Always be precise.}  \minghao{fixed}
\item $G_n, G^{\pp,\qqq}_n$ denote the \emph{random geometric graph} $G(X_1, \cdots, X_n; r(n))$ and its \emph{$(\pp,\qqq)$-perturbation}, respectively. 
\end{itemize}
For any graph $G$, let $V(G)$ and $E(G)$ refer to its vertex set and edge set, and let $N_{G}(u)$ denote the set of neighbors of $u$ in $G$ (i.e. nodes connected to $u \in V(G)$ by edges in $E(G)$). For a subset $W \subseteq \mathbb{R}^d$, we denote the number of indices $i \in \{1,\cdots, n\}$ such that $X_i \in W$ by $\mathcal{N}(W) = \mathcal{N}_n(W)$; that is,  $\mathcal{N}(W)$ is the number of points from $\mathcal{X}_n = \{ X_1, \ldots, X_n\}$ contained in $W$.
%that is, $\mathcal{N}(W) = \sum\limits_{i=1}^n 1_W(X_i)$. 
%We denote the volume of $W$ by $vol(W)$. Finally let $\theta$ denote the volume of the unit ball $B_1(0)$; that is $\theta = vol(B_1(0))$. 

%everything is set in Euclidean space: a fixed positive integer $d$ and a fixed norm $\norm{\cdot}$ on $\mathbb{R}^d$ are given; $\myBC$ is the \Bconst{} of $(\mathbb{R}^d, \norm{\cdot})$ (see Section \ref{sec:well_separated_family}); $\nu$ is a probability distribution with finite maximum density $\sigma$; $X_1, X_2, \cdots$ are independent random variables sampled from $\nu$; $r = \{r(n)\}_n$ is a sequence of positive real numbers such that $r(n) \rightarrow 0$ as $n \rightarrow \infty$; $p$ and $q=q(n)$ are real numbers between 0 and 1 (for simplicity, we only consider the case when $p$ is a fixed constant); and for $n = 1, 2, \cdots,$ the notations $G_n, G^{p,q}_n$ denote the \emph{random geometric graph} $G(X_1, \cdots, X_n; r(n))$ and its \emph{$(p,q)-$perturbation}, respectively. For any graph $G$, let $V(G)$ and $E(G)$ refer to its vertex set and edge set, and let $N_{G}(u)$ denote the set of neighbors of $u$ in $G$ (i.e. nodes connected to $u \in V(G)$ by edges in $E(G)$). For some set $W \subseteq \mathbb{R}^d$, we denote the number of indices $i \in \{1,\cdots, n\}$ such that $X_i \in W$ by $\mathcal{N}(W) = \mathcal{N}_n(W)$; that is, $\mathcal{N}(W) = \sum\limits_{i=1}^n 1_W(X_i)$. We denote the volume of $W$ by $vol(W)$. Finally denote $\theta$ to be the volume of the unit ball $B_1(0)$; that is $\theta = vol(B_1(0))$. 

\subsection{Overview of main results}\label{sec:main_results}
We now state the main results of this paper, which concerns the behavior of clique number of $(\pp,\qqq)$-perturbed noisy random geometric graphs. 
To understand the behavior of $\pp$-deletion and $\qqq$-insertions (which have different effects on the clique numbers), we first separate the {\it insertion-only} case (where the perturbation only has random insertions) and the {\it deletion-only} case (where the perturbation only has random deletions), and present results for the two cases in Theorem \ref{thm:insertion_only_RGG} and \ref{thm:deletion_only_RGG}, respectively. 

\myparagraph{\bf Insertion only.~} 
We first consider the clique number of $G_n^{0,\qqq}$, where no edges in $G_n$ are removed and only new edges are added. The graph generated this way can be thought of as the union of a random geometric graph and an \myER{} random graph. Indeed, in Theorem \ref{thm:insertion_only_RGG} below, we show the interplay between those two random graphs as $\qqq = \qqq(n)$ increases in different regimes of $r$.
%Our main result in this part is the following theorem.

\begin{theorem}\label{thm:insertion_only_RGG}
Given a $(0,\qqq)$-perturbed noisy random geometric graph $G_n^{0,\qqq}$ in \mysetting{}, the following holds: 
%as in section \ref{subsec:definitions_notations}, the following hold.

\begin{enumerate}\denselist
\item[(\rom{1})] Suppose that $nr^d \leq n ^{-\alpha}$ for some fixed $\alpha \in \left( 0, 1 / \myBC^2 \right]$. Then there exist constants $C_1, C_2$ such that
\begin{itemize}\denselist
\item[(\rom{1}.a)] if $\qqq \leq \left(1 / n \right)^{C_1}$, then a.s.
$$
\omega\left(G_n^{0,\qqq}\right) \sim  1,$$
%\item $\exists$ a constant $0 < C_3 < 1$ such that if $\left(\frac{1}{n}\right)^{\xi} \ll q \leq C_3 < 1$ for all $\xi > 0$, then
\item[(\rom{1}.b)] and if $\left(1 / n\right)^{C_1} < \qqq \leq C_2$, then a.s.
$$
\omega\left(G_n^{0,\qqq}\right) \sim \log_{ 1/ \qqq}{n} .
$$
\end{itemize}

\item[(\rom{2})] Suppose that for every $\epsilon > 0$, $n^{-\epsilon} \ll nr^d \ll \log{n}$. Then there exist constants $C_1, C_2$ such that
\begin{itemize}\denselist
\item[(\rom{2}.a)] if $\qqq \leq \left(nr^d / \log{n} \right)^{C_1}$, then a.s.
$$\omega \left( G_n^{0,\qqq} \right)\sim \dfrac{\log{n}}{\log{( \log{n} / nr^d)} },$$
%~~~~\text{a.s.}
%\item $\exists$ a constant $0 < C_3 < 1$ such that if $\left(\frac{nr^d}{\log{n}}\right)^{\xi} \ll q \leq C_3 < 1$ for all $\xi > 0$, then

\item[(\rom{2}.b)] and if $\left( nr^d / \log n \right)^{C_1} < \qqq \leq C_2$, then a.s.
$$\omega\left(G_n^{0,\qqq}\right) \sim \log_{1 / \qqq}{n}.$$
\end{itemize}

\item[(\rom{3})] Suppose that $\sigma n r^d / \log n  \rightarrow t \in (0, \infty)$ as $n \to \infty$. Then
there exists a constant $C_1$  such that if $\qqq \leq C_1$, then a.s.
$$\omega\left(G_n^{0,\qqq}\right) \sim  nr^d. $$
\end{enumerate}
\end{theorem}

\medskip

 For this insertion-only case, one could view the graph $G_n^{0,\qqq}$ as the union of a random geometric graph $G_n$ and an \myER{} random graph $G(n, \qqq)$ on the same vertex set. The theorem above suggests that intuitively either the clique number from the random geometric graph or the one from the \myER{} random graph will dominate, depending on regimes of $nr^d$ and $\qqq$. 
 On the surface, this may not look surprising. 
However, in general, the clique number of the union $G = G_1 \cup G_2$ of two graphs $G_1$ and $G_2$ could be significantly larger than the clique number in each individual graph $G_i$: Consider for example $G_1$ is a collection of $\sqrt{n}$ disjoint cliques, each of size $\sqrt{n}$, while $G_2$ equals to the complement of $G_1$. The union $G_1 \cup G_2$ is the complete graph and the clique number is $n$. However, the clique number of $G_1$ or $G_2$ is $\sqrt{n}$. Our results suggest that due to the randomness in each of the individual graph we are considering, with high probability such a scenario will not happen and the two types of random graph do not interact strongly. 

To handle the mixture of random geometric graph with inserted edges, it is not clear how to use classical tools such as scan statistic. One of the key ideas in our paper is to develop and use what we call a \emph{a well-separated clique-partitions family} (see section \ref{sec:well_separated_family}) to help us to decouple the interaction between the two types of hidden random structures (i.e, random geometric graph, and the $(0,\qqq)$-perturbation). We believe that this idea is interesting for its own sake.

%To prove our technical results, \newchange{one of the key ideas in our approach is to develop and use what we call a \emph{a well-separated clique-partitions family}}
%we apply a novel approach using what we call a well-separated clique-partitions family 
%(see section \ref{sec:well_separated_family}) to help us to decouple the interaction between the two types of hidden random structures (i.e, random geometric graph, and the $(0,\qqq)$-perturbation).

%\subsubsection{Deletion only}\label{subsubsec:deletion_only}
\myparagraph{\bf Deletion only.~} We now present our main result for the clique number of $G_n^{\pp, 0}$, where we only delete edges in $G_n$ with a fixed edge-deletion probability $\pp$. We remark that technically speaking, the deletion-only case is easier to handle than the insertion-only case. 
%Our main result in this part is the following theorem.

\begin{theorem}\label{thm:deletion_only_RGG}
Given a $(\pp, 0)$-perturbed noisy random geometric graph $G_n^{\pp,0}$ in \mysetting{} with a fixed constant $0 < \pp < 1$, the following holds: 

\begin{enumerate}\denselist
\item[(\rom{1})] Suppose that $nr^d \leq n ^{-\alpha}$ for some fixed $\alpha > 0$. Then a.s.
\begin{align*}
\omega\left(G_n^{\pp,0}\right) \sim 1
\end{align*}
\item[(\rom{2})] Suppose that $n^{-\epsilon} \ll nr^d \ll \log{n}$ for all $\epsilon > 0$. Then a.s.
\begin{align*}
\omega\left(G_n^{\pp,0}\right) \sim \log \frac{\log n}{\log (\log n / nr^d)} 
\end{align*}
\item[(\rom{3})] Suppose that $\sigma n r^d / \log n \rightarrow t \in (0, \infty)$. Then a.s.
\begin{align*}
\omega\left(G_n^{\pp,0}\right) \lesssim  \log\left(nr^d\right)
\end{align*}
Furthermore, there exists a constant $T > 0$ such that if $\sigma nr^d \geq T\log n$, then a.s.
\begin{align*}
\omega\left(G_n^{\pp,0}\right) \sim \log\left(nr^d\right)
\end{align*}
\end{enumerate}
\end{theorem}

\paragraph{\bf Combined case.~}
The above {\it insertion-only} and {\it deletion-only} cases in fact represent key technical challenges. 
When there are both random insertions and deletions, we can derive some bounds on the clique number by simply combining the above results and some technical lemmas later in the paper together with the monotone property of clique number. For example, we have the following result when $nr^d$ is in the subcritical regime.
\begin{cor}\label{cor:combined_RGG}
Given a $(\pp,\qqq)$-perturbed noisy random geometric graph $G_n^{\pp,\qqq}$ in \mysetting{} with a fixed constant $0 < \pp < 1$ and suppose that $nr^d \leq n ^{-\alpha}$ for some fixed $\alpha \in \left( 0, 1 /\myBC^2\right]$, then there exist constants $C_1, C_2$ such that
\begin{itemize}\denselist
\item[a)] if $\qqq \leq \left(1/n\right)^{C_1}$, then a.s.
\begin{align*}
\omega\left(G_n^{\pp,\qqq}\right) \sim 1
\end{align*}
\item[b)] and if $\left(1 / n\right)^{C_1} < \qqq \leq C_2$, then a.s.
\begin{align*}
\omega\left(G_n^{\pp,\qqq}\right) \sim \log_{1 /\qqq}{n}
\end{align*} 
\end{itemize}
\end{cor}

The complete list of results can be found in Theorem \ref{thm:combined_RGG} of Section \ref{sec:combined_case}.  
%The complete result for all regimes can be found in Theorem \ref{thm:combined_RGG} in Section \ref{sec:combined_case}.

%\subsection{Preliminaries}
\section{Preliminaries and well-separated clique-partitions family} 
In this section, we first state in Section \ref{sec:standardtools} some existing results / tools that we will use frequently throughout the paper. 
We will then define in Section \ref{sec:well_separated_family} a new object called \emph{well-separated clique-partitions family} which we will need later in the arguments.  

\subsection{Some standard results and tools} 
\label{sec:standardtools}

For the proofs in this paper, we need some bounds on the binomial and Poisson distributions. 
\begin{lemma}[Lemma 3.6 in \cite{mcdiarmid2011chromatic}]\label{lem:binomial_upperbound}
Let $Z $ be either binomial or Poisson and $k \geq \mu := \myE[Z]$. Then
\begin{align*}
\left(\frac{\mu}{ek}\right)^k \leq \myprob[Z \geq k] \leq \left(\frac{e\mu}{k}\right)^k
\end{align*}
\end{lemma}

%\yusu{Minghao: is $p$ below constant? If so we should say ``$p\in (0, 1)$ constant", no? Or is it standard in the literature that $p\in (0,1)$ means that it is a constant? Also, make sure that we use this consistently throughout the paper.} 
\begin{lemma}[Chernoff--Hoeffding theorem \cite{penrose2003random}]\label{lem:chernoff_hoeffding}
Suppose $n \in \mathbb{N}, \alpha \in (0,1)$ and $0 < k < n$. Let $X \sim \text{Bin}(n, \alpha)$ be either a binomial random variable with mean $\mu = n\alpha$ or $X \sim Poisson(\mu)$ be a Poisson random variable with mean $\mu > 0$. %Then we have: 
%\begin{enumerate}
%\item[(\rom{1})] 
If $k \geq \mu$, then
%\begin{align*}
$$\myprob[X \geq k] \leq \exp \left( -\mu H\left(\frac{k}{\mu}\right) \right)$$ 
%\end{align*}
%\item[(\rom{2})] if $k \leq \mu$, then
%\begin{align*}
%$\myprob[X \leq k] \leq \exp \left( -\mu H\left(\frac{k}{\mu}\right) \right),$
%\end{align*}
%\end{enumerate}
where $H: [0, \infty] \rightarrow [0, \infty)$ is a function defined by $H(0)=1$ and 
%\begin{align*}
$H(a) = 1 -a + a \log a.$
%\end{align*}
\end{lemma}

One object we use frequently in our proofs is the following \emph{(generalised) scan statistics} defined in \cite{mcdiarmid2011chromatic}. Recall that for any set $U \subseteq \mathbb{R}^d$, $\mathcal{N}(U)$ is the number of points from $\mathcal{X}_n = \{ X_1, \ldots, X_n\}$ contained in $U$.
%\yusu{Minghao, I removed several your math environment (like ``align"): I think as long as they are clear, they dont have to always be in a separate line. Sometimes one may want them separate line as one may want to refer to it later. But otherwise, as long as the presentation / display look clear, there is no need to over use them. }
%\yusu{Minghao: are you sure that is the definition of $rW$? I thought $rW$ is the $r$-thickening of $W$: that is, $\{x\in \mathbb{R}^d \mid d(x, W) \le r\}$. For your definition, it depends on the relative position of $W$ to the origin, and the diameter of $rW$ can be arbitrarily bigger than $W$.} \minghao{The definition here is correct, although some results only holds for special $W$'s. Here $W$ can be any set even not connected. However, later we will only pick some good sets such as $W = B_1(\origin)$, so everything will be fine. I copied this definition from \cite{mcdiarmid2011chromatic}. Maybe we should only say $W$ instead of that $rW$?}
\begin{definition}[(Generalised) scan statistic \cite{mcdiarmid2011chromatic}]\label{def:generalised_scan_statistic}
For any set $W \subseteq \R^d$, we define $M_{W}$ by
%\begin{align*}
$M_W:= \max_{x \in \mathbb{R}^d}\mathcal{N}(x+rW)$ where $rW = \{rw: w \in W\}$ is the scaled set of $W$.
%\end{align*}
\end{definition}
In other words, $M_W$ is the maximum number of points in $\mathcal{X}_n$ in any translate of $rW$.

\subsection{Well-separated clique-partitions family}
\label{sec:well_separated_family}
This section discuss one main technique we will use to bound the clique number of $G_n^{\pp,\qqq}$ from above. The main challenge here is to disentangle two types of randomness --- the location of vertices and the $(\pp,\qqq)$-perturbation. %Unlike soft random geometric graphs, where the probability of connecting vertices $u$ and $v$ is a decreasing function of the distance $\norm{u-v}$ \cite{waxman1988routing},
In particular, our model allows vertices even far away to each other to become connected. To solve this, we develop a novel approach using what we call a \emph{well-separated clique-partitions family} (to be defined shortly) to help us to decouple the interaction between these two types of hidden random structures. 

To set up the stage, we first recall the \emph{Besicovitch covering lemma} which has a lot of applications in measure theory \cite{federer2014geometric}.  
\begin{definition}[Packings, covers, and partitions]\label{def:packings}

(1) A \emph{packing} is a countable collection $\mathcal{B}$ of \emph{pairwise disjoint} closed balls in $\R^d$. 
Such a collection $\mathcal{B}$ is a \emph{packing w.r.t.  a set $P\subset \mathbb{R}^d$} %for a subset $P \subset X$ 
if the centers of the balls in $\mathcal{B}$ lie in the set $P$, and it is a \emph{$\delta$-packing} if all of the balls in $\mathcal{B}$ have radius $\delta$. 

(2) A set $\{A_1, \ldots, A_\ell\}, A_i\subseteq \mathbb{R}^d$, \emph{covers $P$} if $P \subseteq \bigcup_i A_i$. 

(3) Given a set $A$, we say that $A$ is \emph{partitioned into} $A_1, A_2, \cdots, A_k$, if $A = A_1\cup \cdots \cup A_k$ and $A_i \cap A_j = \emptyset$ for any $i\neq j$. 
\end{definition}

\begin{lemma}[Besicovitch Covering Lemma \cite{heinonen2012lectures}]\label{thm:Besicovitch}
There exists a constant $\myBC = \myBC(d) \in \mathbb{N}$ such that for any set $P \subset \mathbb{R}^d$ and $\delta >0$,  there are $\myBC$ number of $\delta$-packings w.r.t. $P$, denoted by $\{\mathcal{B}_1, \cdots, \mathcal{B}_\myBC\}$, whose union also covers $P$.
\end{lemma}

We call the constant $\myBC$ above the \emph{Besicovitch constant}. Note that $\myBC$ depends only on the dimension $d$ and is {\bf not} dependent of $\delta$. 
%$\delta$ represents the scale where events are allowed to take place, and it could be large.
%Given a set $A$, we say that $A$ is \emph{partitioned into} $A_1, A_2, \cdots, A_k$, if $A = A_1\cup \cdots \cup A_k$ and $A_i \cap A_j = \emptyset$ for any $i\neq j$. 

%\yusu{Minghao, it is a little confusing as below, vertex set for $G_n$ is $V$, while when we define it, we used $X_n$. I assume this would happen later too. We need a consistent way.} 
\begin{definition}[Well-separated clique-partitions family]\label{def:wellseparated}
Given a geometric graph $G^*$ in $(\mathbb{R}^d, \norm{\cdot})$ with vertex set $V$ and edge set $E$, a family $\mathcal{P} = \{P_i\}_{i \in \Lambda}$, where $P_i\subseteq V$ and $\Lambda$ is the index set of $P_i$s, forms a \emph{\myWSP}{} of $G^*$ if: 
\begin{enumerate}\denselist
\item $V = \cup_{i \in \Lambda} P_{i}$.
\item $\forall i \in \Lambda$, $P_i$ can be partitioned as $P_i= C_1^{(i)} \sqcup C_2^{(i)}\sqcup \cdots \sqcup C_{m_{i}}^{(i)}$ where
\begin{enumerate}
\item[(2-a)] $\forall j \in [1, m_i]$, there exist $\bar{v}^{(i)}_j \in V$ such that $C_j^{(i)} \subseteq B_{r/2}\left(\bar{v}^{(i)}_j\right) \cap V$. 
\item[(2-b)] For any $j_1, j_2 \in [1, m_{i}]$ with $j_1 \neq j_2$, $d_H\left(C_{j_1}^{(i)}, C_{j_2}^{(i)}\right) >r$, where $d_H$ is the Hausdorff distance between two sets in $\mathbb{R}^d$ with respect to norm $\norm{\cdot}$.
\end{enumerate}
\end{enumerate} 
We also call $C_1^{(i)} \sqcup C_2^{(i)}\sqcup \cdots \sqcup C_{m_{i}}^{(i)}$ a \emph{\cliqueP{} of $P_i$ (w.r.t. $G^*$)}, and its \emph{size} (cardinality) is $m_i$. 
The \emph{size} of the \myWSP{} $\mathcal{P}$ is its cardinality $|\mathcal{P}| = |\Lambda|$. 
\end{definition}

In the above definition, (2-a) implies that each $C_{j}^{(i)}$ spans a clique in the geometric graph $G^*$; thus we call $C_j^{(i)}$ as a \emph{clique} in $P_i$ and $C_1^{(i)} \sqcup C_2^{(i)}\sqcup \cdots \sqcup C_{m_{i}}^{(i)}$ a \cliqueP{} of $P_i$. (2-b) means that there are no edges in $G^*$ between any two cliques of $P_{i}$. As a result, any edge in its corresponding $(\pp,\qqq)$-perturbation $G^{\pp,\qqq}_n$ between such cliques must come from $(\pp,\qqq)$-perturbation ($\qqq$-insertion). We will leverage this fact significantly later when bounding clique numbers. See figure \ref{fig:wellseparated}. 

\begin{figure}[htbp]
\centering
\includegraphics[height=6cm]{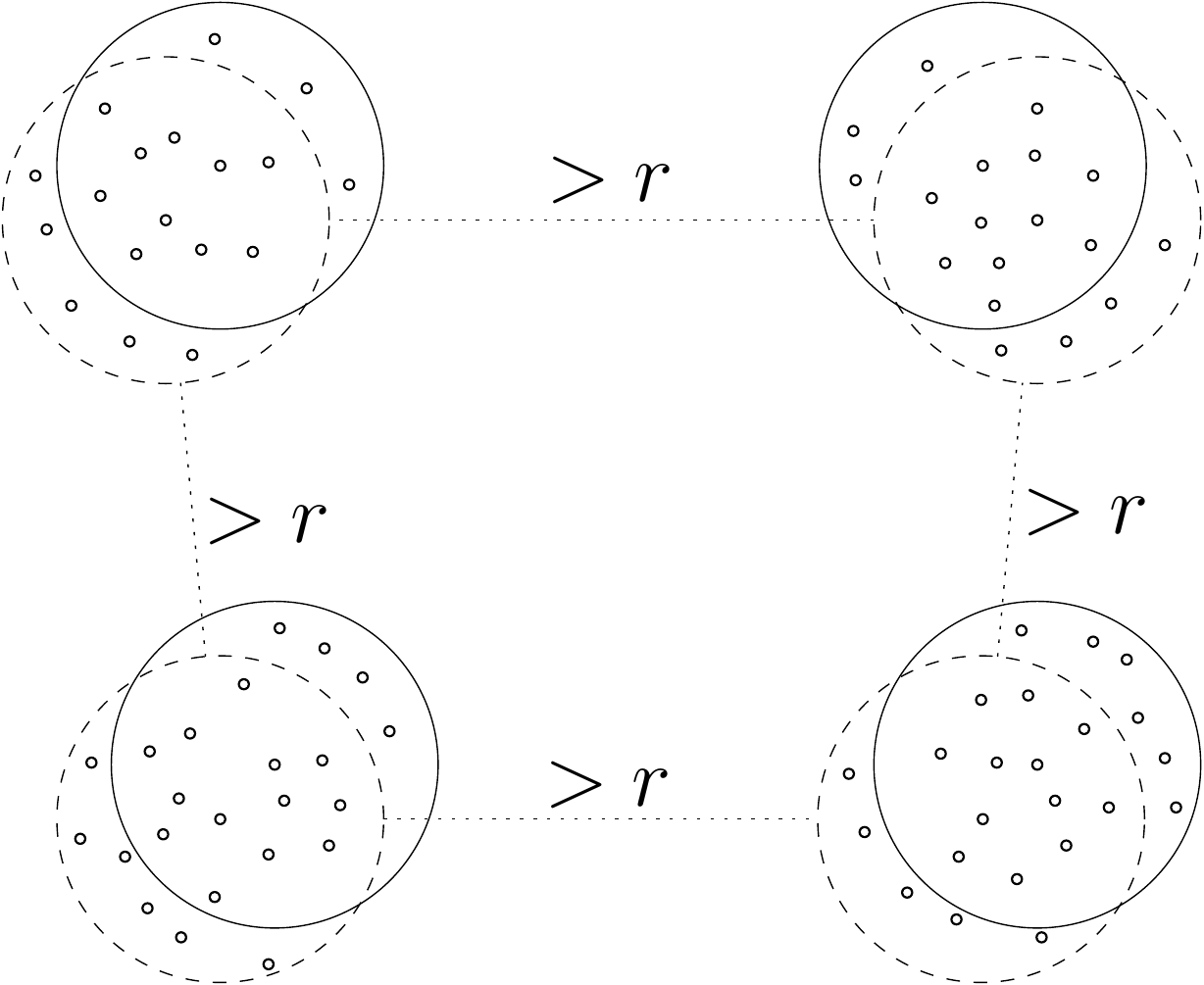}
\caption{Points in the solid balls are $P_1$, and those in dashed balls are $P_2$. Each adapts a \cliqueP{} of size $m_1 = m_2 = 4$. 
Assuming that all nodes in $G^*$ are shown in this figure, then $\mathcal{P} = \{P_1, P_2\}$ forms a \myWSP{} of $G^*$.}
\label{fig:wellseparated}
\end{figure}

We have the following existence theorem of a well-separated clique-partitions family with constant size only depending on the dimension $d$.

\begin{theorem}\label{lem:BCLdoubling}
There exists a \myWSP{} $\mathcal{P}=\{P_i\}_{i \in \Lambda}$ of any geometric graph $G^*$ with $|\Lambda| \leq \myBC^2$, where $\myBC = \myBC(d)$ is the \Bconst{} of $\mathbb{R}^d$. 
\end{theorem}
\begin{proof}%[Proof of Theorem \ref{lem:BCLdoubling}]
%\yusu{Minghao: $X$ should be changed to $\mathbb{R}^d$ throughout the paper, if we have defined to use $\mathbb{R}^d$, no? We should remark in Concluding remarks that this can be extended. (I changed in this proof, but please go over the paper and change in other places.)} \minghao{fixed.}
To prove the theorem, first imagine we grow an $r/2$-ball around each node in $V \subset \mathbb{R}^d$ (the vertex set of $G^*$). 
By the Besicovitch Covering Lemma (Lemma \ref{thm:Besicovitch}), we have a family of $(r/2)$-packings w.r.t. $V$, $\mathcal{B} = \{\mathcal{B}_1, \mathcal{B}_2, \cdots, \mathcal{B}_{\alpha_1}\}$, whose union covers $V$. Here, the constant $\alpha_1$ satisfies $\alpha_1 \le \myBC(d)$. 

Each $\mathcal{B}_i$ contains a collection of disjoint $r/2$-balls centered at a subset of nodes in $V$, and let $V_i \subseteq V$ denote the centers of these balls. For any $u, v \in V_i$, we have $\norm{u-v} > r$ as otherwise, $B_{r/2}(u) \cap B_{r/2}(v) \neq \emptyset$ meaning that the $r/2$-balls in $\mathcal{B}_i$ are not all pairwise disjoint. 
Now consider the collection of $r$-balls centered at all nodes in $V_i$. 
Applying Besicovitch Covering Lemma to $V_i$ again with $\delta = r$, we now obtain a family of $r$-packings w.r.t. $V_i$, denoted by $\mathcal{D}^{(i)} = \mathcal{D}_1^{(i)} \sqcup \cdots \sqcup \mathcal{D}_{\alpha_2^{(i)}}^{(i)}$, whose union covers $V_i$. Here, the constant $\alpha_2^{(i)}$ satisfies $\alpha_2^{(i)}\le \myBC(d)$ for each $i \in [1, \alpha_1]$. 

Now each $\mathcal{D}_{j}^{(i)}$ contains a set of disjoint $r$-balls centered at a subset of nodes $V_j^{(i)} \subseteq V_i$ of $V_i$. 
First, we claim that $\bigcup_j V_j^{(i)} = V_i$. This is because that $\mathcal{B}_i$ is an $r/2$-packing which implies that $\norm{u-v} > r$ for any two nodes $u, v \in V_i$. In other words, the $r$-ball around any node from $V_i$ contains no other nodes in $V_i$. 
As the union of $r$-balls $ \mathcal{D}_1^{(i)} \sqcup \cdots \sqcup \mathcal{D}_{c_2^{(i)}}^{(i)}$ covers $V_i$ by construction, it is then necessary that each node $V_i$ has to appear as the center in at least one $\mathcal{D}_j^{(i)}$ (i.e, in $V_j^{(i)}$). Hence $\bigcup_j V_j^{(i)} = V_i$.

Now for each vertex set $V_j^{(i)}$, let $P_j^{(i)} \subseteq V$ denote all points from $V$ contained in the $r/2$-balls centered at points in $V_j^{(i)}$. As $\bigcup_j V_j^{(i)} = V_i$, we have $\bigcup_j P_j^{(i)} = \bigcup_{v\in V_i} \left(B_{r/2}(v) \cap V\right)$. It then follows that $\bigcup_{i \in [1, \alpha_1]}\left( \bigcup_{j\in [1, \alpha_2^{(i)}]} P_j^{(i)} \right) = V$ as the union of the family of $r/2$-packings $\mathcal{B} = \{\mathcal{B}_1, \mathcal{B}_2, \cdots, \mathcal{B}_{c_1}\}$ covers all points in $V$ (recall that $\mathcal{B}_i$ is just the set of $r/2$-balls centered at points in $V_i$). 

Clearly, each $P_j^{(i)}$ adapts a \cliqueP{}: Indeed, for each $V_j^{(i)}$, any two nodes in $V_j^{(i)}$ are at least distance $2r$ apart (as the $r$-balls centered at nodes in $V_j^{i}$ are disjoint), meaning that the $r/2$-balls around them are more than distance $r$  away, in the Hausdorff metric. 
In other words, $\mathcal{P} = \left\{ P_j^{(i)}, i\in [1, \alpha_1], j \in [1, \alpha_2^{(i)}] \right\}$ forms a \myWSP{} of $G^*$.  
Finally, since $\alpha_1, \alpha_2^{(i)} \le \myBC(d) = \beta$, the cardinality of $\mathcal{P}$ is thus bounded by $\beta^2$. 
\end{proof}

%%%%%%%%%%%%%%%%%%%%%%%%%%%%%%%%%%%%%%%%%%%%%%%%%%%%%%%%%%%%%%%%%%%%%%%%%%%%%%%%%%%%%%%%%%%%%%%%%%%%%%%%%%%%%%

\section{Proof of Theorem \ref{thm:insertion_only_RGG}}\label{sec:proof_insertion_only_RGG}
In this section, we focus on estimating the order of $\omega\left(G_n^{0,\qqq} \right)$, the clique number of $G_n^{0,\qqq}$. Note that for any set $W \subseteq \mathbb{R}^d$, the generalised scan statistic $M_W$ (see Definition \ref{def:generalised_scan_statistic}) is the maximum number of points in the vertex set $\mathcal{X}_n = \{X_1, X_2, \cdots, X_n\}$ in any translate of $rW$. Set $W_{1/2} := B_{1/2}(\origin)$ and $W_1 := B_{1}(\origin)$ where $\origin$ is the origin. 
It is obvious that $\omega\left(G_n^{0,\qqq}\right) \geq M_{W_{1/2}}$. Thus, the lower bound can be directly derived by using the results related to the generalised scan statistic in \cite{mcdiarmid2011chromatic}. However, getting an upper bound is much more challenging, since unlike $\omega(G_n) \leq M_{W_2}$ holds in $G_n$, the vertices in a clique of $G_n^{0,\qqq}$ can come from everywhere in the space.

\subsection{Proof of Part (\rom{1}) --- subcritical regime}\label{sec:insertion_only_very_sparse}
In this section, we discuss the order of $\omega(G_n^{0,\qqq})$ in the regime $nr^d \leq n ^{-\alpha}$ for some fixed $\alpha \in \left(0, 1 / \myBC^2\right]$. First, we define the \emph{long-edges} in $G_n^{0,\qqq}$.

\begin{definition}[long-edges]\label{def:longedge}
An edge $(u,v)$ in a $(0,\qqq)-$perturbed noisy random geometric graph $G^{0,\qqq}_n$ is a \emph{long-edge} if and only if $\norm{u-v} > 3r$.
\end{definition}

Cliques in $G^{0,\qqq}_n$ can be classified into the following two types: 
\begin{itemize}\denselist
\item Type-\rom{1} clique: doesn't contain any long-edges 
\item Type-\rom{2} clique: contains at least one long-edge
\end{itemize}

In what follows, we derive upper bounds for each type of cliques separately. Intuitively, Type-\rom{1} cliques primarily depend on the underlying random geometric graph, while Type-\rom{2} sees a stronger effect of the \myER{}-perturbation. We use $3r$ as the threshold in the definition of long-edges, to intuitively decouple the interaction of the local neighborhood of $u$ and $v$ within the random geometric graphs.
The lower bounds are easier to derive and can be found later in Section \ref{subsubsec:remaining}. 
\paragraph{\bf Type-\rom{1} cliques.~}
%\subsubsection{Type-\rom{1} cliques} 
The case of Type-\rom{1} cliques is rather simple to handle:  
Note that vertices of one Type-\rom{1} clique are contained within a ball of radius $3r$ centered at some vertex $X_i\in \mathcal{X}_n$. 
Thus, to bound the size of such clique from above, it suffices to estimate the number of vertices in each of the $n$ number of $3r$-ball centered at some vertex in $\mathcal{X}_n$. Set $W_3 := B_{3}(\origin)$. We have the following lemma, which gives a uniform upper bound of number of vertices in any $3r-$ball. It is a simplified variant of Lemma 3.8 in \cite{mcdiarmid2011chromatic}. We include its simple proof for completeness.

\begin{lemma}\label{lem:num_points_3rball_very_sparse}
If $nr^{d} \leq n^{-\alpha}$ then $\myprob\left[M_{W_3} \leq \ceil*{4 / \alpha}\right] = 1 + O(n^{-3})$.
\end{lemma}
\begin{proof}%[Proof of Lemma \ref{lem:num_points_3rball_very_sparse}]
For some fixed integer $k$, we have the following inequality. 
\begin{align*}
\myprob\left[M_{W_3} \geq k+1\right] &\leq \myprob\left[\exists i: \mathcal{N}\left(B_{6r}(X_i)\right) \geq k + 1\right] 
~~ \leq~~ n \myprob\left[ \mathcal{N}\left(B_{6r}(X_1)\right) \geq k + 1\right]
\end{align*}
Furthermore, note that 
\begin{align*}
\myprob\left[ \mathcal{N}\left(B_{6r}(X_1)\right) \geq k + 1\right] &\leq \myprob\left[Bin \left(n, \sigma \theta (6r)^d  \right) \geq k \right]
~\leq~\left( \frac{e\sigma \theta 6^d (nr)^d}{k}\right)^{k} =~ O(n^{-k\alpha})
\end{align*}
%\yusu{Minghao, what is $\sigma$? I added $\theta$. It would also be good if you use $\cdot$ sometimes to remove ambiguity; otherwise, it is confusing whether $\theta(???)$ is $\theta$ times $(???)$ or $\theta$ is a function. } 
%\yusu{Why the change from $k+1$ to $k$?}
Recall that $\sigma$ is the maximum density of $\nu$ and $\theta = \int_{B_1(\origin)} dx$ are introduced in the standard setting-R at the end of Section \ref{subsec:definitions_notations}. 
The second inequality holds due to Lemma \ref{lem:binomial_upperbound}. Now pick $k = \ceil*{4 / \alpha}$. We then have $\myprob\left[M_{W_3} \leq \ceil*{4 /\alpha}\right] = 1 + O(n^{-3})$ as required.
\end{proof}

It then follows that the size of Type-\rom{1} cliques can be bounded from above by $\ceil*{4 /\alpha}$ almost surely.
%\yusu{Minghao, where are the lower bounds? If it follows from that $M_{1/2}$ said earlier, then you need to state it explicitly, and point out that we only prove upper bounds from now on. It is not clear right now. }

\paragraph{\bf Type-\rom{2} cliques.~}
%\subsubsection{Type-\rom{2} cliques}
Now let's consider the Type-\rom{2} cliques, which is significantly more challenging to handle. Recall that $W_1 = B_{1}(\origin)$. We can use the same argument in Lemma \ref{lem:num_points_3rball_very_sparse} to derive the following lemma which gives a uniform upper bound of the number of points in any $r$-ball. 
\begin{lemma}\label{lem:upperbound_r_ball_very_sparse}
If $nr^{d} \leq n^{-\alpha}$ then $\myprob\left[M_{W_1} \leq \ceil*{4 / \alpha}\right] = 1 + O\left(n^{-3}\right)$.
\end{lemma}

We now introduce a local version of the clique number called \emph{edge clique number}. 
%It can be used to describe the maximum size of the type-\rom{2} cliques.
\begin{definition}[Edge clique number] \label{def:edgecliquenumber}
Given a graph $G = (V,E)$, for any edge $(u,v) \in E$, its \emph{edge clique number $\omega_{u,v}(G)$} is defined as
\begin{align*}
\omega_{u,v}(G) = \text{ the largest size of any clique in $G$ containing } (u,v).
\end{align*}
\end{definition}

We are now ready to bound the size of all the type-\rom{2} cliques in $G_n^{0,\qqq}$. 
More precisely, the following theorem first bound the edge clique number for all \emph{long edge $(u,v)$}. This is the key theorem in this Section, and we include its proof in the next subsection. 

\begin{theorem}\label{thm:type_2_very_sparse}
Given an $(0,\qqq)$-perturbed noisy random geometric graph $G_n^{0,\qqq}$ in \mysetting{} and suppose that $nr^d \leq n ^{-\alpha}$ for some fixed $\alpha \in \left(0, 1 / \myBC^2\right]$, then
\begin{itemize}\denselist 
\item[(a)] There exist constants $C_1, C_2 > 0$ which depend on the \Bconst{} $\myBC$ and $\alpha$ such that if 
\begin{align}\label{eqn:type_2_qbound_very_sparse}
\qqq \leq C_1 \left(1/n\right)^{C_2}
\end{align}
 then, with high probability, for all long-edge $(u,v)$ in $G_n^{0,\qqq}$, its edge clique number $\omega_{u,v}(G_n^{0,\qqq}) \lesssim 1$. 
\item[(b)] There exists a constant $\xi > 0$ which depends on the \Bconst{} $\myBC$ and $\alpha$ such that if $\left(1 / n \right)^{\xi} \leq \qqq < 1$, then, with high probability, for all long-edge $(u,v)$ in $G_n^{0,\qqq}$, its edge clique number $\omega_{u,v}(G_n^{0,\qqq}) \lesssim \log_{1 /\qqq}n$.  
\end{itemize}

\end{theorem}

\subsubsection{Proof of Theorem \ref{thm:type_2_very_sparse}}
\label{subsec:thm:type_2_very_sparse}

%\begin{proof}[Proof of Theorem \ref{thm:type_2_very_sparse}] 

\paragraph{\bf Proof of part (a) of Theorem \ref{thm:type_2_very_sparse}.~} Given any vertex $y$, let $B^{\mathcal{X}_n}_r(y) \subseteq \mathcal{X}_n$ denote $B_r(y) \cap \mathcal{X}_n$. 
Now consider a long-edge $(u,v)$. Set $A_{uv} = \mathcal{X}_n \setminus \left( B^{\mathcal{X}_n}_r(u)\cup B^{\mathcal{X}_n}_r(v) \right)$ and $B_{uv} = B^{\mathcal{X}_n}_r(u)\cup B^{\mathcal{X}_n}_r(v)$. Denote $\tilde{A}_{uv} = A_{uv} \cup \{u\} \cup \{v\}$; 
easy to check that $\mathcal{X}_n = \tilde{A}_{uv} \cup {B}_{uv}$. 

Let $G|_S$ denote the subgraph of $G$ spanned by a subset $S$ of its vertices. 
Given any set $\aC$, let $\aC|_S  = \aC \cap S$ be the restriction of $\aC$ to another set $S$. 
Now consider a subset of vertices $\aC \subseteq \mathcal{X}_n$: obviously, $\aC = \aC|_{\tilde{A}_{uv}}\cup \aC|_{B_{uv}}$. 

Set $N_{\max} := \ceil*{4 / \alpha}$. Denote $\aF$ to be the event that ``for every $v \in \mathcal{X}_n$, the ball $B_{r}(v)\cap \mathcal{X}_n$ contains at most $N_{\max}$ points''; 
 and $\aF^c$ denotes the complement event of $\aF$. By Lemma \ref{lem:upperbound_r_ball_very_sparse}, we know that, $\myprob[\aF^c] = O(n^{-3})$.

Let $\aK \geq 8 \myBC^2$ be an integer to be determined. By applying the pigeonhole principle and the union bound, we have: 
\begin{align}\label{eqn:mainresult_very_sparse}
&\myprob\left[ \omega_{u,v}\left(G_n^{0,\qqq}\right) \ge \aK  \middle| \aF \right] \nonumber \\
\le ~~ &\myprob\left[ \omega_{u,v}\left( G_n^{0,\qqq}|_{\tilde{A}_{uv}} \right) \ge \aK /2\middle| \aF\right] + \myprob\left[ \omega_{u,v}\left( G_n^{0,\qqq}|_{B_{uv}} \right) \ge \aK /2 \middle| \aF\right] 
\end{align}
Next, we bound the two terms on the right hand side of Eqn. (\ref{eqn:mainresult_very_sparse}) separately in Case (i) and Case (ii) below. 

\begin{figure}[h]
\centering
\begin{tabular}{ccc}
\includegraphics[height=3.5cm]{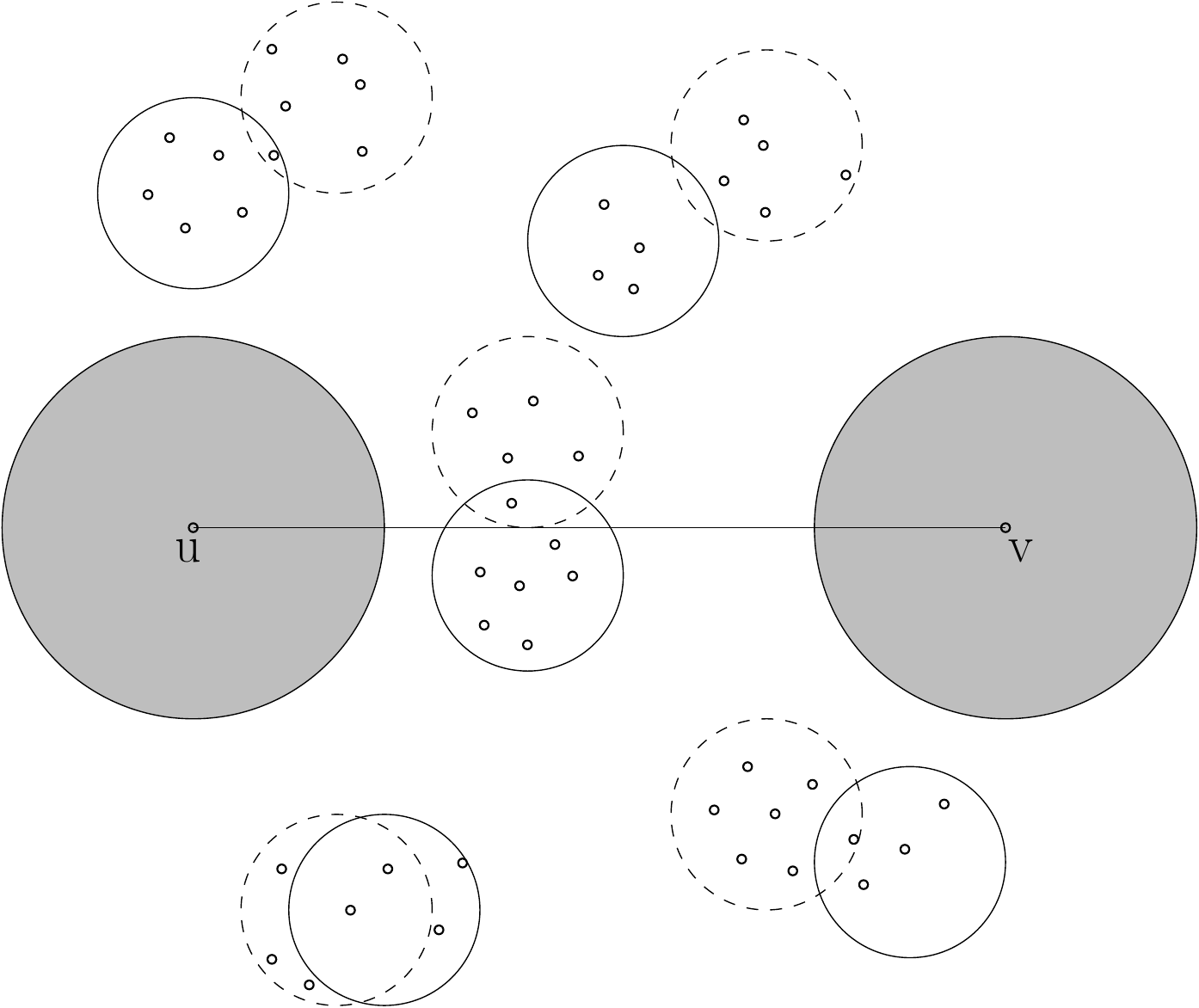} &\hspace*{0.3in}  & \begin{minipage}[c]{6cm}
\vspace*{-3.5cm}
\includegraphics[height=2cm]{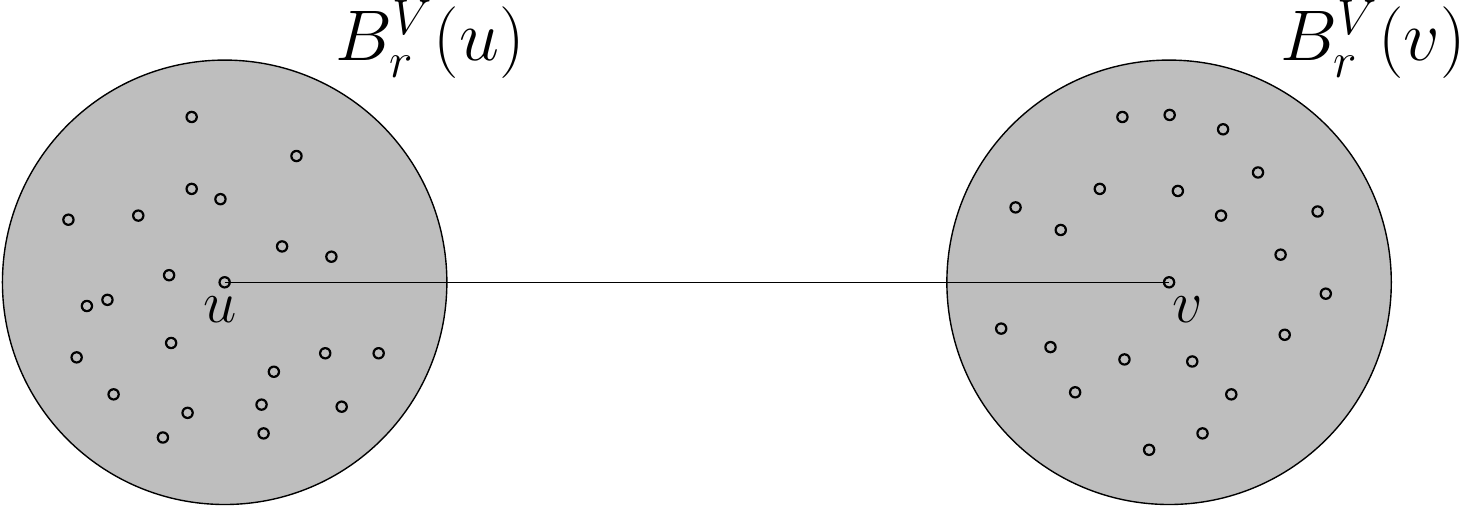}
\end{minipage}  \\
(a) &  & \hspace*{0in}(b)
\end{tabular}
\vspace*{-0.15in}
\caption{(a) A well-separated clique partition $\mathcal{P} = \{P_1, P_2\}$ of $A_{uv}$ --- points in the solid balls are $P_1$, and those in dashed balls are $P_2$. (b) Points in $B_{uv}$. }
\label{fig:twocases_very_sparse}
\end{figure}

\paragraph{\bf Case (\romannumeral 1): bounding the first term in Eqn. (\ref{eqn:mainresult_very_sparse}).~} 
Directly bounding the edge clique number using points from $\tilde{A}_{uv}$ is challenging, as two types of edges are involved (a ``local" edge from random geometric graph, or a randomly inserted \myER{} type edge). Hence we will use the \myWSP{} introduced earlier, to consider only special types of cliques where the number of combinatorial choices these two types of edges can induce is limited.
We apply Theorem \ref{lem:BCLdoubling} for points in $A_{uv}$.   
This gives us a \myWSP{} $\mathcal{P} = \{P_i\}_{i\in \Lambda}$ of $A_{uv}$ with $|\Lambda| \le \myBC^2$ being a constant. 
See Figure \ref{fig:twocases_very_sparse} (a).
Augment each $P_i$ to $\tilde{P}_i = P_i \cup \{u\} \cup \{v\}$. 
Suppose there is a clique $C$ in $G_n^{0, \qqq}|_{\tilde{A}_{uv}}$, then as $\bigcup_i \tilde{P}_i = \tilde{A}_{uv}$, we have $C = \bigcup_{i\in \Lambda} C|_{\tilde{P}_i}$, implying that $|C| \le \sum_{i\in \Lambda} \left| C|_{\tilde{P}_i}\right|$. 
Hence again by applying pigeonhole principle and the union bound, we derive the following inequality: 
\begin{align}\label{eqn:caseAunionbound_very_sparse}
&\myprob\left[ \omega_{u,v} \left(G_n^{0,\qqq}|_{\tilde{A}_{uv}} \right) \ge \aK / 2 \middle| \aF\right] 
~~\leq~~ \sum_{i=1}^{|\Lambda|} ~\myprob\left[ \omega_{u,v}\left(G_n^{0,\qqq}|_{\tilde{P}_{i}}\right) \ge \aK / (2|\Lambda|) \middle| \aF\right]
\end{align}    

Now for arbitrary $i \in \Lambda$, consider  $G_n^{0,\qqq}|_{\tilde{P}_i}$, the induced subgraph of $G_n^{0,\qqq}$ spanned by vertices in $\tilde{P}_i$. Note, $G_n^{0,\qqq}|_{\tilde{P}_i}$ can be viewed as generated by inserting each edge not in $G_n|_{\tilde{P}_i} \cup \{(u,v)\}$ with probability $\qqq$. 
Recall from Definition \ref{def:wellseparated} that each $P_i$ adapts a clique-partition $C_1^{(i)}\sqcup \dots \sqcup C_{m_i}^{(i)}$, where every $C_j^{(i)}$ is contained in an $r/2$-ball, and all such balls are $r$-separated (w.r.t Hausdorff distance).

Now fix any $i \in \Lambda$. For simplicity of the argument below, set $m = m_i$, and let $N_j = \left|C_j^{(i)}\right|$ denote the number of points in the $j$-th cluster $C_j^{(i)}$. 
Note that obviously, $m \le |P_i| \le \left|\mathcal{X}_n \right|=n$ for any $i\in \Lambda$. We also know that if event $\aF$ has already happened, then $N_j \leq N_{\max}$.

Observe that the induced subgraph $G_n^{0,\qqq}|_{\tilde{P}_i}$ consists of a set of cliques (each clique is spanned by some $C_j^{(i)}$ with edges coming from the base random geometric graph $G_n$), $u$, $v$, edge $(u,v)$, and inserted edge between them with insertion probability $\qqq$ (see Figure \ref{fig:illustration_case_A}).

\begin{figure}[htbp]
\centering
\includegraphics[height=6cm]{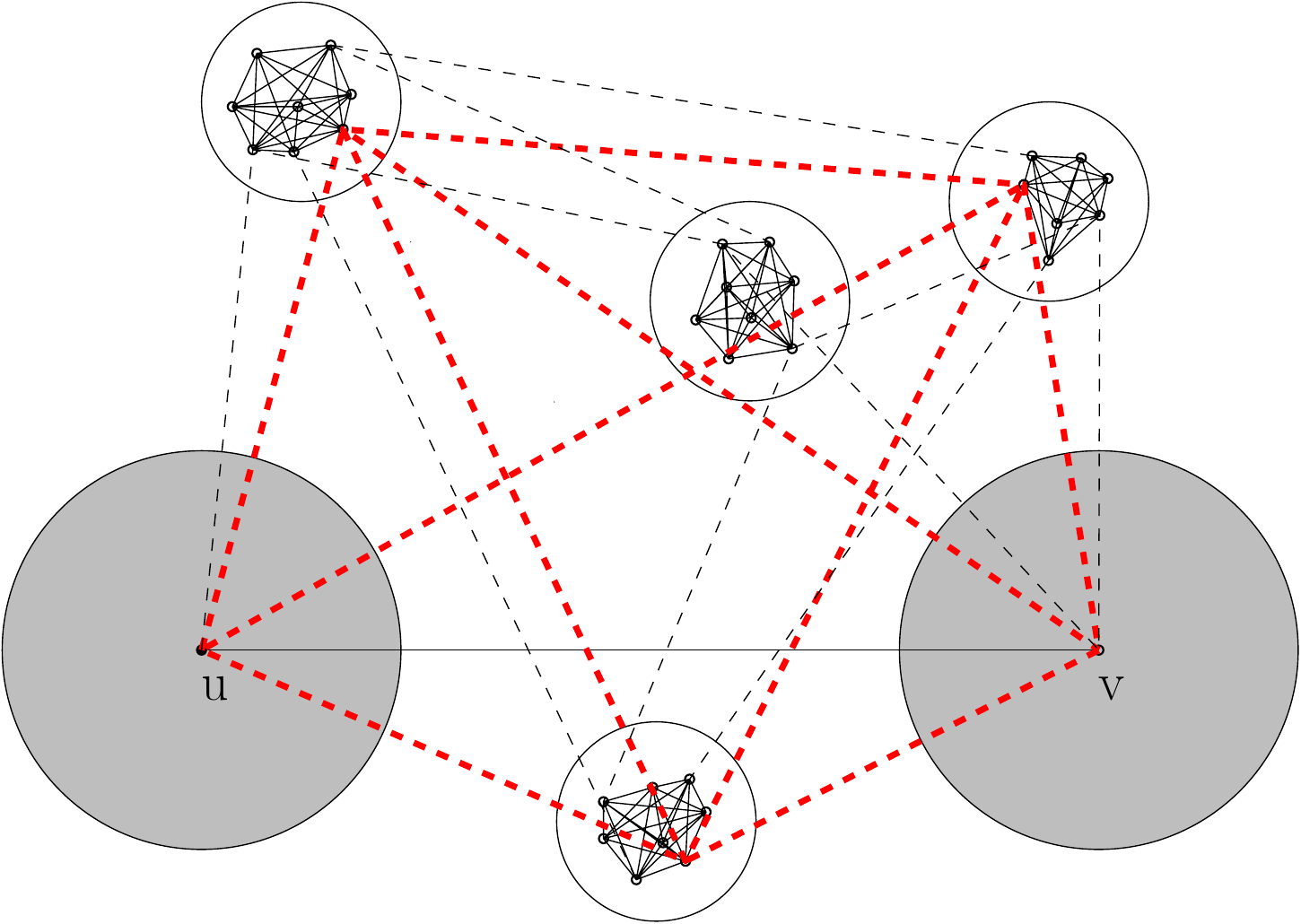}
\caption{The red dashed lines and the edge $uv$ form a possible clique in some well-separated clique partition $P_i$. The points in the small balls are the nodes falling in $r/2-$balls (and thus they are all pairwise connected in the base random geometric graph $G_n$). All the dashed lines are the randomly inserted edges (independently with probability $\qqq$).}
%(all non-existing edges are inserted independently with probability $q$).}
\label{fig:illustration_case_A}
\end{figure}

Now set $k := \floor*{\aK / 2|\Lambda|}-2$. Since $\aK \geq 8 \myBC^2$, easy to see that $k \geq 1$. For every set $S$ of $k+2$ vertices in this graph $G_n^{0,\qqq}|_{\tilde{P}_i}$, let $A_S$ be the event ``\emph{$S$ is clique in $G_n^{0,\qqq}|_{\tilde{P}_i}$ containing $(u,v)$ given $\aF$}'' and $\aI_S$ its indicator random variable. Set 
\begin{align*}
\aI = \sum\limits_{|S|=k+2}\aI_S
\end{align*}
and note that $\aI$ is the number of cliques of size $(k+2)$ in $G_n^{0,\qqq}|_{\tilde{P}_i}$ containing $(u,v)$ given $\aF$. 
It follows from Markov inequality that:  
\begin{align}\label{eqn:probGPi_very_sparse}
\myprob\left[\omega_{u,v} \left(G_n^{0,\qqq}|_{\tilde{P}_{i}}\right) \ge  k+2 \middle| \aF\right] = \myprob[\aI > 0] \leq \myE[\aI] 
\end{align}
On the other hand, using linearity of expectation, we have:
\begin{align}\label{eqn:expectationX_very_sparse}
\myE[\aI] ~~=~~ \sum\limits_{|S|=k+2}\myE[\aI_S] ~~=~~ \qqq^{2k}\sum\limits_{\substack{x_1 + x_2 + \cdots + x_m = k \\ 0 \leq x_i \leq N_i}} \binom{N_1}{x_1} \binom{N_2}{x_2}\cdots \binom{N_m}{x_m}\qqq^{(k^2-\sum_{i=1}^{m}x_i^2)/2}\nonumber\\
\leq \qqq^{2k}\sum\limits_{\substack{x_1 + x_2 + \cdots + x_m = k \\ 0 \leq x_i \leq N_{\max}}} \binom{N_{\max}}{x_1} \binom{N_{\max}}{x_2}\cdots \binom{N_{\max}}{x_m}\qqq^{(k^2-\sum_{i=1}^{m}x_i^2)/2}
\end{align}

To estimate this quantity, we have the following lemma:  
\begin{lemma}\label{lem:insertioncaseA_very_sparse}
If $1 \leq k \leq N_{\max}$ and $\qqq$ is less than or equal to
\begin{align}\label{eqn:qbound1_very_sparse}
\min \left\{ \frac{1}{\sqrt{e}} \left(\frac{1}{n^3m}\right)^{\frac{1}{2k}} \left( \frac{k}{N_{\max}}\right)^{\frac{1}{2}}, \frac{1}{2ek^{\frac{1}{k}}}\left(\frac{1}{n^3m^2} \right)^{\frac{1}{k}} \frac{k}{N_{\max}}, \frac{1}{e^{\frac{4}{k}}}\left(\frac{1}{n^3m^k} \right)^{\frac{4}{k^2}}\left(\frac{k}{N_{\max}} \right)^{\frac{4}{k}} \right\},
\end{align}
then we have that $\myE[\aI] = O(n^{-3})$.
\end{lemma}

The proof of this lemma is rather technical, and can be found in Appendix \ref{appendix:lem:insertioncaseA_very_sparse}. 

Note that $\alpha \leq  1 / \myBC^2$, thus $2 N_{\max} = 2\ceil*{4 /\alpha} \geq 8 \myBC^2$. Note that if $\aK \in \left[8 \myBC^2, 2N_{\max}\right]$, then it is easy to check that the assumption $1 \leq k \leq N_{\max}$ in Lemma \ref{lem:insertioncaseA_very_sparse} holds.

Furthermore, $|\Lambda| \le \myBC^2$ (which is a constant) and $m = |P_i| \le |\mathcal{X}_n| = n$. 
One can then verify that there exist constants $c^a_1$ and $c^a_2$ (which depend on the \Bconst{} $\myBC$ and $\alpha$), such that if 
\begin{align*}
\qqq \le c^a_1 \cdot \left(1 / n\right)^{c^a_2/\aK},
\end{align*}
then the conditions in Eqn. (\ref{eqn:qbound1_very_sparse}) will hold. %(the simple proof of this can be found in Appendix \ref{appendix:constantsc1c2_very_sparse}). \yusu{Minghao, see my comment in the appendix: just to verify that $c_1^a$ and $c_2^a$ do not depend on $\alpha$ either, right? they are purely constants {\bf independent} of any constants of the region or $\beta$, right?} \minghao{They depend on $\alpha$. Talk to you tomorrow. The $\aK$ and $k$ in this regime are all finite, but for other regimes, they must go to infinity as $n$ goes to infinity.} 
Thus, combining this with Lemma \ref{lem:insertioncaseA_very_sparse} and Eqn. (\ref{eqn:probGPi_very_sparse}), we know that
%\yusu{Minghao: I think here $\forall i \in \Lambda$ should be outside here, before ``If $8\beta^2$...", no? } 
\begin{align}
\text{If}~~8\myBC^2 \leq \aK \leq 2N_{\max} ~&\text{and}~~ \qqq \leq c^a_1 \cdot \left( 1 / n\right)^{c^a_2/\aK}, \nonumber \\
%,~~\text{then}~ 
&\text{then}~~ \forall i \in \Lambda, \myprob\left[ \omega_{u,v}\left(G_n^{0,\qqq}|_{\tilde{P}_{i}}\right) \ge  k+2\middle| \aF\right] = O(n^{-3}). \label{eqn:qone_very_sparse}
\end{align}

On the other hand, note that
\begin{align*}
&\myprob\left[\omega_{u,v}\left(G_n^{0,\qqq}|_{\tilde{P}_{i}}\right) \ge \aK / (2|\Lambda|)\middle| \aF\right] ~~=~~ \myprob\left[ \omega_{u,v}\left(G_n^{0,\qqq}|_{\tilde{P}_{i}}\right) \ge  k+2\middle| \aF\right]
\end{align*}

As $|\Lambda|$ is a constant, by Eqn. (\ref{eqn:caseAunionbound_very_sparse}), we obtain that
\begin{align}
\text{if}~~\forall i \in &\Lambda,  \myprob\left[\omega_{u,v}\left(G_n^{0,\qqq}|_{\tilde{P}_{i}}\right) \ge \aK / (2|\Lambda|)\middle| \aF\right] = O(n^{-3}) \text{, ~then}~ \nonumber \\
&\myprob\left[\omega_{u,v}\left(G_n^{0,\qqq}|_{\tilde{A}_{uv}}\right) \ge \aK / 2\middle| \aF\right] = O(|\Lambda| n^{-3}) = O(n^{-3}). \label{eqn:qtwo_very_sparse}
\end{align}

It then follows from Eqn. (\ref{eqn:qone_very_sparse}) and (\ref{eqn:qtwo_very_sparse}) that 
\begin{align}\label{eqn:caseAalmostfinal_very_sparse}
\text{If}~ 8\myBC^2 \leq \aK \leq 2N_{\max} &~\text{and}~ \qqq \leq c^a_1 \cdot \left(1 /n \right)^{c^a_2/\aK},\nonumber\\
 &~\text{then}~\myprob\left[ \omega_{u,v}\left(G_n^{0,\qqq}|_{\tilde{A}_{uv}}\right) \ge \aK / 2\middle| \aF\right] = O(n^{-3}). 
\end{align}
Finally, suppose $\aK > \aK_0 = 2N_{\max}$.  
%(and the upper bound assumption on $q$ remains the same). 
Using Eqn (\ref{eqn:caseAalmostfinal_very_sparse}), we know that if $\qqq \le c^a_1 \cdot \left(1 /n\right)^{c^a_2/\aK_0}$ and $\aK > \aK_0$ (in which case note also that $c^a_1 \cdot \left(1/n\right)^{c^a_2/\aK_0} \le c^a_1 \cdot \left(1 /n\right)^{c^a_2/\aK}$), then $$\myprob\left[ \omega_{u,v}\left(G_n^{0,\qqq}|_{\tilde{A}_{uv}}\right) \ge \aK / 2\middle| \aF\right] \le \myprob\left[\omega_{u,v}\left(G_n^{0,\qqq}|_{\tilde{A}_{uv}}\right) \ge \aK_0 / 2\middle| \aF\right] = O(n^{-3}). $$
Combining this with Eqn. (\ref{eqn:caseAalmostfinal_very_sparse}), we thus obtain that: 
\begin{align}\label{eqn:caseAfinal_very_sparse}
\text{If}~ \aK \geq 8\myBC^2 &~\text{and} ~\qqq \leq \min\left\{c^a_1 \cdot \left(1 /n \right)^{c^a_2/(2 N_{\max})} , ~c^a_1 \cdot \left(1 / n\right)^{c^a_2/\aK}\right\}, \nonumber\\
&~\text{then}~\myprob\left[\omega_{u,v}\left( G_n^{0,\qqq}|_{\tilde{A}_{uv}}\right) \ge \aK / 2\middle| \aF\right] = O(n^{-3}). 
\end{align}

%\noindent{\bf Case (\romannumeral 2): bounding the second term in Eqn. (\ref{eqn:mainresult_very_sparse}).} 
\paragraph{\bf Case (\romannumeral 2): bounding the second term in Eqn. (\ref{eqn:mainresult_very_sparse}).~}
First recall that $B_{uv} =  B_r^{\mathcal{X}_n}(u) \cup B^{\mathcal{X}_n}_r(v)$ (see Figure \ref{fig:twocases_very_sparse} (b)). 

On one hand, imagine we now build the following random graph $\tilde{G}_{uv}^{local} = (\tilde{V}, \tilde{E})$: The vertex set $\tilde{V}$ is simply ${B}_{uv}$. To construct the edge set $\tilde{E}$, first, add edges between all pairs of distinct vertices in $B^{\mathcal{X}_n}_r(u)$ and do the same thing for $B^{\mathcal{X}_n}_r(v)$; that is, every two vertices in $B^{\mathcal{X}_n}_r(u)$ or $B^{\mathcal{X}_n}_r(v)$ are now connected by an edge. Next, add edge $(u,v)$. Finally, insert each crossing edge $(x,y)$ with $x\in B^{\mathcal{X}_n}_r(u)$ and $y\in B^{\mathcal{X}_n}_r(v)$ with probability $\qqq$. 

On the other hand, consider the graph $G_n^{0,\qqq}|_{{B}_{uv}}$, the induced subgraph of $G_n^{0,\qqq}$ spanned by vertices in ${B}_{uv}$. We can imagine that the graph $G_n^{0,\qqq}|_{B_{uv}}$ was produced by first taking the induced subgraph $G_n|_{{B}_{uv}}$, and then insert crossing edges $(x,y)$ each with probability $\qqq$. Since $(u,v)$ is a long-edge, by Definition \ref{def:longedge}, we know that there are no edges between nodes in $B^{\mathcal{X}_n}_r(u)$ and $B^{\mathcal{X}_n}_r(v)$ in $G_n|_{{B}_{uv}}$. Since every two vertices in $B^{\mathcal{X}_n}_r(u)$ or $B^{\mathcal{X}_n}_r(v)$ are not necessarily connected by an edge in $G_n|_{{B}_{uv}}$, we know that
\begin{align}\label{eqn:insertiononlycaseBfirststep_very_sparse}
&\myprob\left[\omega_{u,v}\left(G_n^{0,\qqq}|_{B_{uv}}\right) \ge \aK / 2 \middle| \aF\right] 
~~\leq~~ \myprob\left[\omega_{u,v}\left(\tilde{G}_{uv}^{local}\right)\ge \aK / 2\middle| \aF \right]
\end{align}    

Using a similar argument as in case (\romannumeral 1) (the missing details can be found in Appendix \ref{appendix:detailscaseBinsertiononly_very_sparse}), we have that there exist constants $c_1^b,c_2^b >0$ which depend on the Besicovitch constant $\myBC$ and $\alpha$ such that
\begin{align*}
\text{If} ~\aK \geq 8\myBC^2 &~\text{and}~ \qqq \leq ~ c^b_1 \cdot \left(1 / n\right)^{c^b_2/\aK},\nonumber\\ &~\text{then} ~ \myprob \left[\omega_{u,v}\left(\tilde{G}_{uv}^{local}\right) \ge \aK /2 \middle| \aF \right] = O(n^{-3})
\end{align*}

Pick $\aK = 2N_{\max}  = 2 \ceil{4 /\alpha} \geq 8\myBC^2$ (by condition $\alpha \in \left(0, 1 /\myBC^2 \right]$). Note that $\aK = O(1)$ in this case. Thus, combining the above bound with Eqn. (\ref{eqn:insertiononlycaseBfirststep_very_sparse}), (\ref{eqn:caseAfinal_very_sparse}) and (\ref{eqn:mainresult_very_sparse}), there exist constants $C_1 = \min\{c_1^a, c_1^b\}$ and $C_2 = \max\{c^{a}_2, c^{b}_2\}$ such that if $\qqq$ satisfies conditions in Eqn. (\ref{eqn:type_2_qbound_very_sparse}), then 
\begin{align*}
\myprob\left[\omega_{u,v}\left(G_n^{0,\qqq}\right) \ge \aK\right] \leq \myprob\left[\omega_{u,v}\left(G_n^{0,\qqq}\right) \ge \aK \middle| \aF\right] + \myprob[\aF^c]  = O(n^{-3})
\end{align*}

Finally, by applying the union bound, this means:
\begin{align*}
\myprob\left[ \text{for all long-edge } (u,v) \text{, } \omega_{u,v}\left(G_n^{0,\qqq}\right) \ge \aK\right]  = O(n^{-1})
\end{align*}
Thus with high probability, we have that for all long-edge $(u,v)$, $\omega_{u,v}(G_n^{0,\qqq}) = O(1)$ as long as Eqn. (\ref{eqn:type_2_qbound_very_sparse}) holds. This completes the proof of Part (a) of Theorem \ref{thm:type_2_very_sparse}.

\myparagraph{\bf Proof of part (b) of Theorem \ref{thm:type_2_very_sparse}. ~} 
We use the same strategy in the proof of part (a). That is, we again try to bound the two terms on the right hand side of Eqn. (\ref{eqn:mainresult_very_sparse}) from above respectively. The key difference here is to give an alternative estimate of Eqn. (\ref{eqn:expectationX_very_sparse}) in case (\romannumeral 1) and its counterpart in case (\romannumeral 2) under the new constraint of $\qqq$.

For case (\romannumeral 1), instead of using Lemma \ref{lem:insertioncaseA_very_sparse}, we now use the following lemma, whose proof can be found in Appendix \ref{appendix:lem:insertioncaseA_very_sparse_looser}.

\begin{lemma}\label{lem:insertioncaseA_very_sparse_looser}
There exists a constant $C_3 > 0$ depending on the \Bconst{} $\myBC$ and $\alpha$ such that if $\left(1 / n \right)^{8 / (3N_{\max})} \leq \qqq < 1$ and $\aK = C_3 \floor*{\log_{1 / \qqq}n}$, then we have that $\myE[\aI] = O(n^{-3})$.
\end{lemma}
%The proof of this lemma can be found in Appendix \ref{appendix:lem:insertioncaseA_very_sparse_looser}.

Now choose such $C_3$ as specified in Lemma \ref{lem:insertioncaseA_very_sparse_looser}. We know that the following holds.
\begin{align}\label{eqn:caseAalmostfinal_very_sparse_looser}
\text{If}~ &\left(1 / n \right)^{8 / (3N_{\max})} \leq \qqq < 1, \nonumber\\
&\text{then}~\myprob\left[ \omega_{u,v}\left(G_n^{0,\qqq}|_{\tilde{A}_{uv}}\right) \ge C_3\floor*{\log_{1 /\qqq}n} / 2\middle| \aF\right] = O(n^{-3}). 
\end{align}

For case (\romannumeral 2), we know that if event $\aF$ has already happened, then $|B_{uv}| \leq 2N_{\max}$, where $|B_{uv}|$ denotes the cardinality of set $B_{uv}$. Note that if $\left(1 /n \right)^{C_3 / (4N_{\max}+C_3)} \leq \qqq < 1$, then
$$C_3\floor*{\log_{1 /\qqq}n} / 2 \geq 2N_{\max} \geq |B_{uv}|.$$ Hence, we obtain that:
\begin{align}\label{eqn:caseBalmostfinal_very_sparse_looser}
\text{If}~ &\left(1 /n \right)^{C_3 /(4N_{\max}+C_3)} \leq \qqq < 1, \nonumber\\
&~~\text{then}~\myprob\left[ \omega_{u,v}\left(G_n^{0,\qqq}|_{B_{uv}}\right) \ge C_3\floor*{\log_{1 / \qqq}n} /2\middle| \aF\right] = 0. 
\end{align}
Set $\xi = \min\left\{ 8 /(3N_{\max}), C_3 / (4N_{\max} + C_3)  \right\}$, which is also a constant. Thus, combining Eqn. (\ref{eqn:caseBalmostfinal_very_sparse_looser}), (\ref{eqn:caseAalmostfinal_very_sparse_looser}) and (\ref{eqn:mainresult_very_sparse}), we know that if $\left(1 / n \right)^{\xi} \leq \qqq < 1$, then
\begin{align*}
\myprob\left[\omega_{u,v}\left(G_n^{0,\qqq}\right) \ge C_3\floor*{\log_{1 /\qqq}n} \middle| \aF\right]  = O(n^{-3}).
\end{align*}
Finally, by a similar argument in the proof for Part (a) using the law of total probability and union bound, we can show that with high probability, we have that for any long-edge $(u,v)$, $\omega_{u,v}\left( G_n^{0,\qqq}\right) \lesssim \log_{1 / \qqq}n$. This completes the proof of Theorem \ref{thm:type_2_very_sparse}.
%\end{proof}

\subsubsection{Finishing the proof of Part (I) of Theorem \ref{thm:insertion_only_RGG}}
\label{subsubsec:remaining}~

Based on the discussion of Type-\rom{1} cliques as well as Theorem \ref{thm:type_2_very_sparse} for Type-\rom{2} cliques, we have the following corollary regarding the upper bound of $\omega(G_n^{0,\qqq})$ in the subcritical regime.

\begin{cor}\label{cor:upper_bound_summary_very_sparse}
Given a $(0,\qqq)$-perturbed noisy random geometric graph $G_n^{0,\qqq}$ in \mysetting{} and suppose that $nr^d \leq n ^{-\alpha}$ for some fixed $\alpha \in \left(0, 1 / \myBC^2\right]$, then
\begin{itemize}
\item[(i)] there exist two constants $C_1, C_2 > 0$ such that if $\qqq \leq C_1 \left(1 /n\right)^{C_2}$, then a.s.
\begin{align*}
\omega\left(G_n^{0,\qqq}\right) \lesssim 1
\end{align*}
\item[(ii)] and there exists a constant $\xi > 0$ such that if $\left(1 / n\right)^{\xi} \leq \qqq < 1$, then a.s.
\begin{align*}
\omega\left(G_n^{0,\qqq}\right) \lesssim \log_{1 /\qqq}{n} 
\end{align*} 
\end{itemize}
\end{cor}

Part (i) of Corollary \ref{cor:upper_bound_summary_very_sparse} can be derived by combining Lemma \ref{lem:num_points_3rball_very_sparse} and part (a) of Theorem \ref{thm:type_2_very_sparse}, while part (ii) of Corollary \ref{cor:upper_bound_summary_very_sparse} can be derived by combining Lemma \ref{lem:num_points_3rball_very_sparse} and part (b) of Theorem \ref{thm:type_2_very_sparse}.

%%%%%%%%%%%%%%%%%%%%%%%%%%%%%%%%%%%%%%%%%%%%%%%%%%%%%%%%%%%%%%%%%%%%%%%%%%%%%%%%%%%
To derive a lower bound of $\omega(G_n^{0,\qqq})$, we need the following result on the clique number of \myER{} random graphs (proof can be found in Appendix \ref{appendix:thm:ER_lowerbound}). 

\begin{lemma}\label{thm:ER_lowerbound}
For \myER{} random graph $G(n,p)$ with $\left(1 / n\right)^{1 /11} \leq p \leq (1 /n)^{1 / \sqrt[4]{n}}$, we have a.s. $\omega(G(n,p)) > \floor*{\log_{1/p}{n}}.$
\end{lemma}

Note that $\qqq$ here is no longer a fixed constant as in the standard literature \cites{bollobas2001random, alon2016probabilistic}, thus the well-known $\omega(G(n,p)) \sim 2\log_{1/p}n$ statement cannot be directly applied here. Not surprisingly, the standard second moment method \cite{alon2016probabilistic} is used here, but the calculation is different. Details of the proof can be found in Appendix \ref{appendix:thm:ER_lowerbound}.

Easy to see that $\myprob[\omega(G_n^{0,\qqq}) \ge K] \ge \myprob[\omega(G(n,\qqq)) \ge K] $ for any positive integer $K$. The following corollary of Lemma \ref{thm:ER_lowerbound} gives a lower bound of $G_n^{0,\qqq}$ regardless of which regime $nr^d$ belongs to.

\begin{cor}\label{thm:q_insertion_lowerbound}
Given a $(0,\qqq)$-perturbed noisy random geometric graph $G_n^{0,\qqq}$ in \mysetting{} and suppose that $\left( 1 / n\right)^{1/11} \leq \qqq  \leq (1/n)^{1/\sqrt[4]{n}}$, then we have a.s.
\begin{align*}
\omega\left(G_n^{0,\qqq}\right)  > \floor*{\log_{1 /\qqq}{n}}
\end{align*}
\end{cor}

Note that $(1/n)^{1 /\sqrt[4]{n}} \rightarrow 1$ as $n \rightarrow \infty$. Thus, there exists a constant $C_3 \in (0,1)$ (very close to $1$) such that $C_3 \leq (1 /n)^{1 /\sqrt[4]{n}}$ for sufficiently large $n$. Also note that the following monotone property holds.
\begin{align*}
\text{For any $S > 0$ and $0 \leq q_1 < q_2 < 1$, } \myprob\left[\omega\left(G_n^{0,q_1}\right) \geq S \right] \leq \myprob\left[\omega\left(G_n^{0,q_2}\right) \geq S \right]     
\end{align*}
And by Corollary \ref{cor:upper_bound_summary_very_sparse} (b), we know that $\omega\left(G_n^{0, n^{-\xi}} \right) = O\left(\log_{n^{-\xi}}n\right) = O(1)$ a.s.. Easy to see that there exists a constant $C_1'$ such that $\qqq \leq \left(1 /n\right)^{C_1'}$ implies $\qqq \leq C_1 \left(1 /n\right)^{C_2}$. Also notice that $\log_{1/\qqq}n = \Theta(1)$ for $\qqq \in \left( \left(1 /n\right)^{C_1'}, \left(1 /n\right)^{\xi}\right)$ (if this interval exists). Thus, the lower bound of $\qqq$ in the condition of part (b) of Corollary \ref{cor:upper_bound_summary_very_sparse} (which is $\left(1 / n\right)^{\xi}$) can be extended to $ \left( 1 /n\right)^{C_1'}$ and the conclusion still holds. Combining these facts with Corollary \ref{cor:upper_bound_summary_very_sparse} and Corollary \ref{thm:q_insertion_lowerbound} concludes the proof of part (\rom{1}) of Theorem \ref{thm:insertion_only_RGG}.

%%%%%%%%%%%%%%%%%%%%%%%%%%%%%%%%%%%%%%%%%%%%%%%%%%%%%%%%%%%%%%%%%%%%%%%%%%%%%%%%%%%%%%%%%%%%%%%%%%%%%%%%%%
\subsection{Proof of Part (\rom{2}) --- subcritical regime}
\label{sec:insertion_only_quite_sparse}~

In this section, we discuss the order of $\omega(G_n^{0,\qqq})$ in the regime $n^{-\epsilon} \ll nr^d \ll \log{n}$ for all $\epsilon > 0$. Again, we first derive an upper bound of $\omega(G_n^{0,\qqq})$ by considering two types of cliques (Type-\rom{1} and Type-\rom{2}) introduced in Section \ref{sec:insertion_only_very_sparse}. The idea of the proof in this section is similar to the one in Section \ref{sec:insertion_only_very_sparse}, although the details vary a little. \\

\paragraph{\bf Type-\rom{1} cliques.~}
Recall that $W_3 = B_{3}(\origin)$. The following lemma gives an upper bound of the number of vertices in any $3r-$ball.

\begin{lemma}\label{lem:num_points_3rball_quite_sparse}
If $n^{-\epsilon} \ll nr^d \ll \log{n}$ for all $\epsilon > 0$, then 
\begin{align*}
\myprob\left[M_{W_3} \leq \frac{5\log{n}}{\log \left(\log{n} / (\sigma \theta 6^d nr^d)\right)}\right] = 1 - O(n^{-3}).
\end{align*}
\end{lemma}

The argument is rather standard, relying only on Chernoff--Hoeffding bounds, so we omit the proof. \\

\bigskip

\paragraph{\bf Type-\rom{2} cliques.~}
Recall that $W_1 = B_{1}(\origin)$. By using a similar argument as the one used for Lemma \ref{lem:num_points_3rball_quite_sparse}, we can get the following lemma which gives an upper bound for the number of points in any $r-$ball for the regime of $nr^d$ under discussion.
\begin{lemma}\label{lem:num_points_1rball_quite_sparse}
If $n^{-\epsilon} \ll nr^d \ll \log{n}$ for all $\epsilon > 0$, then 
\begin{align*}
\myprob\left[M_{W_1} \leq \frac{5\log{n}}{\log{\left(\log{n} / (\sigma \theta 2^d nr^d)\right)}}\right] = 1 - O(n^{-3})
\end{align*}
\end{lemma}

Again, the argument is standard, and we omit the proof.\\

\begin{theorem}\label{thm:type_2_quite_sparse}
Given an $(0,\qqq)$-perturbed noisy random geometric graph $G_n^{0,\qqq}$ in \mysetting{} and suppose that $n^{-\epsilon} \ll nr^d \ll \log{n}$ for all $\epsilon > 0$, then
\begin{itemize}\denselist
\item[(a)] there exist constants $C_1, C_2 > 0$ which depend on the \Bconst{} $\myBC$ such that if 
\begin{align}\label{eqn:type_2_qbound_quite_sparse}
\qqq \leq C_1 \cdot \left(nr^d /\log n\right)^{C_2}
\end{align}
 then, with high probability, for all long-edge $(u,v)$ in $G_n^{0,\qqq}$, its edge clique number 
\begin{align*}
\omega_{u,v}(G_n^{0,\qqq}) \lesssim \frac{\log n}{\log \left(\log n / nr^d \right)}. 
\end{align*}
 
\item[(b)] and there exists a constant $\xi$ which depends on the \Bconst{} $\myBC$ such that if $\left(nr^d /\log n \right)^{\xi} \leq \qqq < 1$, then, with high probability, for all long-edge $(u,v)$ in $G_n^{0,\qqq}$, its edge clique number $\omega_{u,v}(G_n^{0,\qqq})  \lesssim \log_{1 /\qqq}n$.  
\end{itemize}

\end{theorem}

The proof of Theorem \ref{thm:type_2_quite_sparse} follows the same flow as the proof of Theorem \ref{thm:type_2_very_sparse}, with some minor changes. The details can be found in Appendix \ref{appendix:thm:type_2_quite_sparse}. \\

\bigskip

\subsubsection{Putting everything together for Part (II) of Theorem \ref{thm:insertion_only_RGG}}~

To wrap up all the above results, we have the following corollary regarding the upper bound of $\omega(G_n^{0,\qqq})$ in subcritical regime.

\begin{cor}\label{cor:upper_bound_summary_quite_sparse}
Given a $(0,\qqq)$-perturbed noisy random geometric graph $G_n^{0,\qqq}$ in \mysetting{} and suppose that $n^{-\epsilon} \ll nr^d \ll \log{n}$ for all $\epsilon > 0$, then
\begin{itemize}
\item[(i)] there exist two constants $C_1, C_2 >0$ such that if $\qqq \leq C_1 \left( nr^d / \log n\right)^{C_2}$, then a.s.
\begin{align*}
\omega\left(G_n^{0,\qqq}\right) \lesssim \frac{\log{n}}{\log{\left(\log{n} / nr^d \right)}}
\end{align*}
\item[(ii)] and there exists a constant $\xi > 0$ such that if $\left(nr^d /\log n\right)^{\xi} \leq \qqq < 1$, then a.s.
\begin{align*}
\omega\left(G_n^{0,\qqq}\right) \lesssim \log_{1 /\qqq}{n}
\end{align*} 
\end{itemize}
\end{cor}

Part (i) can be derived by combining Lemma \ref{lem:num_points_3rball_quite_sparse} and part (a) of Theorem \ref{thm:type_2_quite_sparse}, while part (ii) can be derived by combining Lemma \ref{lem:num_points_3rball_quite_sparse} and part (b) of Theorem \ref{thm:type_2_quite_sparse}.

To derive a tight bound of $\omega(G_n^{0,\qqq})$, in addition to Corollary \ref{thm:q_insertion_lowerbound}, we also need the follwing lemma, which provides a lower bound of $\omega(G_n^{0,\qqq})$.
\begin{lemma}\label{thm:lowerbound_RGG_quite_sparse}
Given a $(0,\qqq)$-perturbed noisy random geometric graph $G_n^{0,\qqq}$ in \mysetting{} and suppose that $n^{-\epsilon} \ll nr^d \ll \log{n}$ for all $\epsilon > 0$, then a.s.
\begin{align*}
\omega\left(G_n^{0,\qqq}\right) \geq \frac{\log{n}}{2\log{\left(\log{n} /nr^d \right)}} 
\end{align*}
\end{lemma}

\begin{proof}[Proof of Lemma \ref{thm:lowerbound_RGG_quite_sparse}]
Note that $\omega\left(G_n^{0,\qqq}\right) \geq M_{W_{1/2}}$ and $M_{W_{1/2}}$ can a.s.\ be bounded from below by $\log{n} / \left(2\log{\left(\log{n} /nr^d \right)}\right).$ (Pick $\epsilon = 1/2$ in Lemma 3.9 of \cite{mcdiarmid2011chromatic}.) 
\end{proof}

Finally, combining Lemma \ref{thm:lowerbound_RGG_quite_sparse} with part (i) of Corollary \ref{cor:upper_bound_summary_quite_sparse} concludes the proof of part (\rom{2}.a) of Theorem \ref{thm:insertion_only_RGG}. Note that there exists a constant $C_1'$ such that $\qqq \leq \left(nr^d /\log n\right)^{C_1'}$ implies $\qqq \leq C_1\left(nr^d / \log n\right)^{C_2}$ and if $\qqq \in \left(\left(nr^d /\log n\right)^{C_1'}, \left(nr^d /\log n\right)^{\xi}\right)$ (if this interval exists), then $$\log_{1/\qqq}n = \Theta\left(\frac{\log{n}}{\log{\left(\log{n} / nr^d \right)}} \right).$$ Thus we can extend the lower bound of the condition in part (ii) of Corollary \ref{cor:upper_bound_summary_quite_sparse} to $\left(nr^d /\log n\right)^{C_1'}$ by the same reasoning at the end of the proof for subcritical regime. Combining these facts with Corollary \ref{thm:q_insertion_lowerbound} and part (ii) of Corollary \ref{cor:upper_bound_summary_quite_sparse} concludes the proof of (\rom{2}.b) of Theorem \ref{thm:insertion_only_RGG}.

%%%%%%%%%%%%%%%%%%%%%%%%%%%%%%%%%%%%%%%%%%%%%%%%%%%%%%%%%%%%%%%%%%%%%%%%%%%%%%%%%%%%%%%

\subsection{Proof of Part (\rom{3}) --- ``supercritical'' regime}\label{sec:insertion_only_dense}~

In this section, we discuss the order of $\omega(G_n^{0,\qqq})$ in the regime $\sigma n r^d / \log n \rightarrow t \in (0, \infty)$. Again, we first derive an upper bound of $\omega(G_n^{0,\qqq})$ by considering two types of cliques (Type-\rom{1} and Type-\rom{2}) introduced in Section \ref{sec:insertion_only_very_sparse}. The idea of the proof in this section is very similar to the one in Section \ref{sec:insertion_only_quite_sparse}, thus the proofs are omitted.

Set $\tau$ be the smallest real number such that $\tau \geq 2$ and $\tau (\log \tau - 1) \geq 4 / (2^d\theta t)$. Since $d,t$ and $\theta$ are all given constants, $\tau$ is also a constant. 
%\yusu{Minghao: So the big $\Theta$ depends on $\frac{1}{t}$ -- is that standard? } 

\paragraph{\bf Type-\rom{1} cliques.~}
The following lemma gives upper bounds of the number of vertices in each $r-$ball and $3r-$ball respectively. Recall that $W_1 = B_{1}(\origin)$ and $W_3 = B_{3}(\origin)$.
%\yusu{Minghao, Lemma \ref{lem:num_points_1rball_dense} is not used since we omitted the proof. It is kind of silly to write this result out (which is also not very important). Why don't you just add the bound for the $r$-ball (from Lemma \ref{lem:num_points_1rball_dense}) also into the lemma below (Lemma \ref{lem:num_points_3rball_dense}). Then remove Lemma \ref{lem:num_points_1rball_dense}.}
\begin{lemma}\label{lem:num_points_3rball_dense}
If $\sigma n r^d / \log n \rightarrow t \in (0, \infty)$, then 
\begin{align*}
\myprob\left[M_{W_1} \leq \tau 2^d \theta \sigma nr^d\right] &= 1 + O(n^{-3})\\
\myprob\left[M_{W_3} \leq \tau 6^d \theta \sigma nr^d\right] &= 1 + O(n^{-3})
\end{align*}
\end{lemma}

\paragraph{\bf Type-\rom{2} cliques.~}
The proof of the following technical theorem is almost the same as the proof of part (a) of Theorem \ref{thm:type_2_quite_sparse} thus is omitted.
%By using a similar argument in Lemma \ref{lem:num_points_3rball_quite_sparse}, we can get the following lemma which gives an upper bound for the number of points in each $r-$ball.
%\begin{lemma}\label{lem:num_points_1rball_dense}
%If $\frac{\sigma n r^d}{\ln n} \rightarrow t \in (0, \infty)$, then 
%\begin{align*}
%\myprob\left[M_{W_1} \leq \tau 2^d \theta \sigma nr^d\right] = 1 + O(n^{-3})
%\end{align*}
%\end{lemma}

\begin{theorem}\label{thm:type_2_dense}
Given a $(0,\qqq)$-perturbed noisy random geometric graph $G_n^{0,\qqq}$ in \mysetting{} and suppose that $\sigma n r^d / \log n \rightarrow t \in (0, \infty)$, then there exists a constant $C$ which depends on the \Bconst{} $\myBC$ such that if  $\qqq \leq C$ then, with high probability, for all long-edge $(u,v)$ in $G_n^{0,\qqq}$, its edge clique number 
$\omega_{u,v}(G_n^{0,\qqq}) \lesssim nr^d$. 
\end{theorem}

To derive a tight bound of $\omega(G_n^{0,\qqq})$, in addition to Corollary \ref{thm:q_insertion_lowerbound}, we also need the following result on lower bound. 
\begin{lemma}\label{thm:lowerbound_RGG_dense}
Given a $(0,\qqq)$-perturbed noisy random geometric graph $G_n^{0,\qqq}$ in \mysetting{} and suppose that $\sigma n r^d / \log n \rightarrow t \in (0, \infty)$, then a.s.
\begin{align*}
\omega\left(G_n^{0,\qqq}\right) \geq \frac{1}{2} \eta \sigma nr^d
\end{align*}
where $\eta$ is the unique solution $x \geq  \theta\left(1/2\right)^d$ to $H\left( x / \left(\theta\left(1/2\right)^d\right)\right) = 1 / \left(\theta\left(1/2\right)^d t\right)$ (recall that function $H$ is defined as $H(a) = 1- a + a\log a$).
\end{lemma}

\begin{proof}[Proof of Lemma \ref{thm:lowerbound_RGG_dense}]
Note that $\omega\left(G_n^{0,\qqq}\right) \geq M_{W_{1/2}}$ and $M_{W_{1/2}}$ can be bounded from below by $\eta \sigma nr^d / 2$ almost surely (directly by Theorem 1.8 of \cite{mcdiarmid2011chromatic}). 
\end{proof}

Finally, combining Lemma \ref{thm:lowerbound_RGG_dense}, Lemma \ref{lem:num_points_3rball_dense} and Theorem \ref{thm:type_2_dense} concludes the proof.

%%%%%%%%%%%%%%%%%%%%%%%%%%%%%%%%%%%%%%%%%%%%%%%%%%%%%%%%%%%%%%%%%%%%%%%%%%%%%%%%%%%%%%%%%%%%%%%%%%%%%%%%%%%%%%

%%%%%%%%%%%%%%%%%%%%%%%%%%%%%%%%%%%%%%%%%%%%%%%%%%%%%%%%%%%%%%%%%%%%%%%%%%%%%%%%%%%%%%%

\section{Proof of Theorem \ref{thm:deletion_only_RGG}}\label{sec:proof_deletion_only_RGG}

In this section, we focus on deriving the order of $\omega\left(G_n^{\pp,0} \right)$. Note that $\omega\left(G_n^{\pp,0} \right) \leq M_{W_{1}}$. Thus, Theorem \ref{thm:deletion_only_RGG} part (\rom{1}) is obvious due to Lemma \ref{lem:upperbound_r_ball_very_sparse}. Our proof of the remaining parts of Theorem \ref{thm:deletion_only_RGG} uses the following lemma following easily from known results in the literature. 
%\yusu{Minghao: what is $\nu$ below? It could be the $\nu$ we use, but also any $\nu$. If you want to state the result only for our case with $\nu$ defined long time back, you need to remind readers what $\nu$ is.} 
%\minghao{fixed.}
Recall that $\nu$ is the probability distribution defined in Section \ref{subsec:definitions_notations}.
\begin{lemma}[Lemma 3.1 in \cite{penrose2003random}]\label{lem:deletion_partition_lemma}
For any fixed $\rho > 0$, recall that $W_1 = B_{1}(\origin)$ (and thus $rW_1 = B_{r}(\origin)$). 
There exists $N = \Omega(r^{-d})$ disjoint translates $x_1 + rW_1, \cdots, x_N + rW_1$ of $rW_1$ with $\nu(x_i + rW_1) \geq (1-\rho)\sigma \theta r^d$ for all $i=1,\cdots, N$.
\end{lemma}
Our proof in this section follows an approach analogous to the proof of Theorem 1.8 in \cite{mcdiarmid2011chromatic}. We show the proof for the subcritical regime here in this section. Since we use similar techniques in the supercritical regime, the proof for that regime (part (\rom{3})) is relegated Appendix \ref{appendix:sec:deletion_only_dense}.

\subsection{Proof of Part (\rom{2}) --- subcritical regime}\label{sec:deletion_only_quite_sparse}
In this section, we discuss the order of $\omega(G_n^{\pp,0})$ in the regime $n^{-\epsilon} \ll nr^d \ll \log{n}$ for all $\epsilon > 0$.

\subsubsection{Deriving upper bound}

We first focus on the upper bound of $\omega\left(G_n^{\pp,0}\right)$. 
This is obtained via considering and relating to the random geometric graphs whose nodes are generated by Poisson point process. %\yusu{Minghao: I added the previous sentence.} 
%\minghao{A \emph{Poisson point process} on $\R^d$ with intensity function $g$ is a point process $\mathcal{P}$ such that for Borel set $A \subseteq \R^d$ the random variable $\mathcal{P}(A)$ is Poisson with parameter $\int_A g(x) dx$ whenever this integral is finite, and if $A_1, \cdots, A_k$ are disjoint Borel sets, then the variables $\mathcal{P}(A_i), 1\leq i \leq k$, are mutually independent.} \minghao{It's quite standard in the area of probability and stochastic process.} 
Let $N \sim Poisson\left((1+\delta)n\right)$ for some $\delta > 0$ (say $\delta = 1 /2$). %\yusu{Minghao: is $\delta$ a fixed constant independent of anything?} \minghao{Not depend on anything. Just some constant.}
%\yusu{Minghao, the next few sentences are not precise: (1) $G_N$ is not defined, so you cannot say that ``Note that $G_N$ is an ...". You can define $G_N$ to be that. (2) what is $f$ in the intensity function?  Please modify here. After you modify it here, you would also need to do it at the beginning for the "Proof of lower bound".} \minghao{fixed.}

Note that $G_N$ (random geometric graph on $N$ nodes; recall Definition \ref{def:RGG}) is a geometric graph ($r$-neighborhood graph) of the Poisson point process $\mathcal{P}_{(1+\delta)n}$ with intensity $(1+\delta)nf$ \cite{penrose2003random}, where $f$ is the density defined in Section \ref{subsec:definitions_notations}. Similar to $G_n^{\pp,0}$, we define a $(\pp,0)-$perturbation of $G_N$ as $G_N^{\pp,0}$. Set $k_n$ be an integer to be determined. Now, we have
\begin{align}
\myprob \left[ \omega\left(G_n^{\pp,0} \right) \geq k_n \right] &\leq \myprob \left[ \omega\left(G_N^{\pp,0} \right) \geq k_n \right] + \myprob \left[N \leq n-1 \right]\nonumber\\
&\leq \myprob \left[ \omega\left(G_N^{\pp,0} \right) \geq k_n \right] + e^{-\gamma n}\nonumber
\end{align}
%\yusu{Minghao: why is the first inequality true?} \minghao{Due to law of total probability.}
for some constant $\gamma > 0$ (depending on $\delta$) %\yusu{Minghao: I added the ``depending on $\delta$", correct?} 
by Lemma \ref{lem:chernoff_hoeffding}. For $y \in \mathbb{R}^d$ let $\mathbf{X}_y$ be the set of nodes of $G_N$ falling in $B_{r}(y)$. Let $M_y$ be the number of points falling in $B_{r}(y)$ spanning a maximum clique in $G_N^{\pp,0}\mid_{\mathbf{X}_y}$. Define $M := \max_{y \in \mathbb{R}^d} M_{y}$. Easy to see
\begin{align}\label{eqn:deletion_only_key_observation_dense}
\myprob \left[ \omega\left(G_N^{\pp,0} \right) \geq k_n \right] &= \myprob\left[M \geq k_n \right]
\end{align}
%\yusu{Why is the above not equality, but inequality?} \minghao{fixed.}
Fix $y \in \mathbb{R}^d$. By the property of Poisson point process, we know $|\mathbf{X}_y| \sim Poisson(\lambda)$ where $\lambda := (1+\delta)n\int_{B_{r}(y)}f(x) dx$.
By using Markov's inequality, we have
\begin{align*}
\myprob\left[M_y \geq k_n \right] &= \sum\limits_{i \geq k_n} \myprob\left[ M_{y} \geq k_n \bigg\vert |\mathbf{X}_y| = i \right] \myprob\left[ |\mathbf{X}_y| = i\right]\\
& = \sum\limits_{i \geq k_n} \myprob\left[\text{number of }k_n\text{-cliques in } G_N^{\pp,0}\mid_{\mathbf{X}_y} \geq 1\bigg\vert |\mathbf{X}_y| = i \right] \frac{e^{-\lambda} \lambda^i}{i!}\\
& \leq \sum\limits_{i \geq k_n} \myE\left[\text{number of }k_n\text{-cliques in } G_N^{\pp,0}\mid_{\mathbf{X}_y} \bigg\vert |\mathbf{X}_y| = i \right] \frac{e^{-\lambda} \lambda^i}{i!}\\
& \leq \sum\limits_{i \geq k_n} \binom{i}{k_n} (1-\pp)^{\binom{k_n}{2}} \frac{e^{-\lambda} \lambda^i}{i!}\\
& = \frac{\lambda^{k_n}}{k_n!}(1-\pp)^{\binom{k_n}{2}} \cdot e^{-\lambda} \sum\limits_{i \geq k_n} \frac{ \lambda^{i-k_n}}{(i-k_n)!}\\
& = \frac{\lambda^{k_n}}{k_n!}(1-\pp)^{\binom{k_n}{2}}
\end{align*}
%\yusu{Minghao: How did you get the last equality in the derivation above?} \minghao{fixed.}

Note that $\lambda \leq (1+\delta)\sigma \theta nr^d$. Thus,
\begin{align*}
\myprob\left[M_y \geq k_n \right] \leq \frac{\left((1+\delta)\sigma \theta nr^d\right)^{k_n}}{k_n!}(1-\pp)^{\binom{k_n}{2}}
\end{align*}
which does not depend on the choice of $y$. Combining this with (\ref{eqn:deletion_only_key_observation_dense}), we have
\begin{align*}
\myprob \left[ \omega\left(G_N^{\pp,0} \right) \geq k_n \right] &\leq \frac{\left((1+\delta)\sigma \theta nr^d\right)^{k_n}}{k_n!}(1-\pp)^{\binom{k_n}{2}}\\
& < \frac{1}{\sqrt{2\pi}} \left(\frac{(1+\delta)e\sigma\theta nr^d (1-\pp)^{(k_n-1)/2}}{k_n} \right)^{k_n}
\end{align*}
Finally, pick 
\begin{align*}
k_n = 2\log_{1 /(1-\pp)} \left(\frac{\log n}{\log \left(\log n /nr^d \right)}\right) + 1.
\end{align*}
Since $n^{-\epsilon} \ll nr^d \ll \log{n}$ for all $\epsilon > 0$, easy to see that $k_n \rightarrow \infty$. Note that 
\begin{align*}
\frac{(1+\delta)e\sigma\theta nr^d (1-\pp)^{(k_n-1)/2}}{k_n} = \frac{(1+\delta)e\sigma\theta }{k_n}\cdot \frac{\log \left(\log n / nr^d \right)}{\log n / nr^d} \leq \frac{C}{k_n}
\end{align*}
for some constant $C >0$. Thus, $\myprob \left[ \omega\left(G_N^{\pp,0} \right) \geq k_n \right] = o(1)$. Hence, we have that a.s.
\begin{align*}
\omega\left(G_n^{\pp,0}\right) \lesssim \log \frac{\log n}{\log \left(\log n / nr^d\right)}
\end{align*}

%\paragraph{Proof of lower bound.} 
\subsubsection{Deriving the lower bound.}

Now we consider the lower bound of $\omega\left(G_n^{\pp,0} \right)$. We first state the following well-known result on the clique number of \myER{} random graphs, which plays an important role in proving the lower bound.
\begin{lemma}\label{lem:ER_lowerbound_well-known}
Suppose $p \in (0,1)$ is a constant. For \myER{} random graph $G(n,p)$ with $n \rightarrow \infty$, we have
\begin{align*}
\myprob\left[\omega(G(n,p)) \leq \floor*{\log_{1 / p}{n}} \right] < e^{-n}
\end{align*}
\end{lemma}
This is a direct corollary of the standard $2\log_{1/p}n$ statement (see P. 185 \cite{alon2016probabilistic}), thus we omit the proof. %\yusu{Minghao: What "the $2\log_{1/p}n$ statement"? Maybe just say before this lemma (instead of after the lemma) sth. like "The lemma below follows easily from classical results on Erdos-Renyi random graphs (e.g, P. 185 [1]). "} \minghao{I will show you the book tomorrow. Roughly speaking, the textbook focuses on deriving an exact value of the clique number, which is $2\log_{1/p}n$, but here I only need $\log_{1/p}n$. The bound here is not the same as the one in the book, but can be easily derived by using the method given in the book.}

Now let $N \sim Poisson\left((1-\delta')n\right)$ for some $\delta' \in (0,1)$ (say $\delta' = 1/2$). Note that $G_N$ is an $r-$neighborhood graph of the Poisson point process $\mathcal{P}_{(1-\delta')n}$ with intensity $(1-\delta')nf$, where $f$ is the density defined in Section \ref{subsec:definitions_notations}. %\yusu{Minghao: Here we have the same issue about $G_N$ and $f$ etc as in the case of upper bound. Modify as well.} \minghao{fixed.}
Similarly, we define a $(\pp,0)-$perturbation of $G_N$ as $G_N^{\pp,0}$. Set $k_n$ be an integer to be determined. Now, we have
\begin{align}
\myprob \left[ \omega\left(G_n^{\pp,0} \right) \leq k_n \right] &\leq \myprob \left[ \omega\left(G_N^{\pp,0} \right) \leq k_n \right] + \myprob \left[N \geq n + 1 \right]\nonumber\\
&\leq \myprob \left[ \omega\left(G_N^{\pp,0} \right) \leq k_n \right] + e^{-\gamma' n}\nonumber
\end{align}
for some constant $\gamma' > 0$ (depending on $\delta'$) by Lemma \ref{lem:chernoff_hoeffding}. Now fix some constant $\rho \in (0,1)$ (say $\rho = 1/2$). Recall $W_{1/2} = B_{1/2}(\origin)$. By Lemma \ref{lem:deletion_partition_lemma}, there exist points $x_1, x_2, \cdots, x_m$ with $m = \Omega\left(r^{-d}\right)$ such that the sets $x_i + W_{1/2}$ are disjoint and $$\nu \left(x_i + W_{1/2} \right) \geq \frac{(1-\rho) \sigma \theta}{2^d} r^d$$ for $i = 1, \cdots, m$ where $\nu$ is the probability distribution defined in Section \ref{subsec:definitions_notations}. 
%\yusu{I think most readers already forgot what $\nu$ is.. you need to remind them. Also, what is relation between $f$ which you used to generate your Poisson point process and this $\nu$?}
%\minghao{fixed.}
Let $\mathbf{X}_i$ be the set of points of $G_N$ falling in $x_i + W_{1/2}$. Then, we have
\begin{align}
\myprob \left[ \omega\left(G_N^{\pp,0} \right) \leq k_n \right] &\leq \myprob \left[ \omega\left(G_N^{\pp,0} \mid_{\mathbf{X}_1} \right) \leq k_n, \cdots, \omega\left(G_N^{\pp,0} \mid_{\mathbf{X}_m} \right) \leq k_n \right]\nonumber\\
& = \prod\limits_{i=1}^{m} \myprob \left[ \omega\left(G_N^{\pp,0} \mid_{\mathbf{X}_i} \right) \leq k_n\right]\label{eqn:deletion_only_large_product}
\end{align}
Note that all the points falling in any $r/2$-ball span a complete graph. Thus, for each $i$, we know the following holds.
\begin{align*}
\myprob \left[ \omega\left(G_N^{\pp,0} \mid_{\mathbf{X}_i} \right) \leq k_n\right] = \myprob \left[ \omega\left(G\left(|\mathbf{X}_i|, 1-\pp\right) \right) \leq k_n\right].
\end{align*}
Set 
\begin{align*}
\Phi_n := \frac{ \log n}{2\log \left(\log n /nr^d\right)}
\end{align*}
which goes to infinty as $n$ grows. Note that $|\mathbf{X}_i| \sim Poisson\left(\tilde{\lambda}\right)$ where 
\begin{align}\label{eqn:lambdabound} 
\frac{(1-\delta')\sigma \theta nr^d}{2^{d}} \geq \tilde{\lambda} := (1-\delta')n \cdot \nu(x_i + W_{1/2}) \geq \frac{(1-\rho)(1-\delta')\sigma \theta nr^d}{2^{d}}.
\end{align}
The upper bound follows from the upper bound of volume of balls. 
%\yusu{(1) Explain how the upper bound is obtained: I undertstand it is from max density and volume of balls. but still. (2) This $\lambda$ is different from the $\lambda$ used in the case for upper bound. It is a little confusing. How about use $\tilde{\lambda}$ here instead? } \minghao{(1) Note that $\sigma$ is the maximum density. So it is just the upper bound of the volume. (2) fixed.}

Now pick $$k_n := \floor*{\log_{1/(1-\pp)}\Phi_n} = \Omega\left( \log \frac{\log n}{\log \left(\log n /nr^d \right)}\right).$$ Let $Q \sim Poisson\left( \tilde{\lambda} / e\right)$. By the law of total probability, we have
\begin{align}
&\myprob \left[ \omega\left(G\left(|\mathbf{X}_i|, 1-\pp\right) \right) \leq k_n\right] \nonumber\\
\leq & \myprob\left[|\mathbf{X}_i| \leq \Phi_n \right] + \sum\limits_{j = \ceil*{\Phi_n}}^{\infty}\myprob \left[ \omega\left(G\left(j, 1-\pp\right) \right) \leq k_n\right] \myprob\left[|\mathbf{X}_i| = j\right]\nonumber\\
 \leq & 1 - \myprob\left[|\mathbf{X}_i| \geq \Phi_n + 1 \right] + \sum\limits_{j = \ceil*{\Phi_n}}^{\infty}\myprob \left[ \omega\left(G\left(j, 1-\pp\right) \right) \leq \floor*{\log_{\frac{1}{1-\pp}}j}\right] \frac{e^{-\tilde{\lambda}} \tilde{\lambda}^j}{j!}\nonumber\\
 < & 1 - \left(\frac{\tilde{\lambda}}{e (\Phi_n + 1)} \right)^{\Phi_n + 1} + \sum\limits_{j = \ceil*{\Phi_n}}^{\infty} e^{-j} \frac{e^{-\tilde{\lambda}} \tilde{\lambda}^j}{j!}\label{eqn:deletion_only_quite_sparse_lower_bound_law_of_total}\\
 = & 1 - e^{-(\Phi_n+1) \log \left(e(\Phi_n + 1) / \tilde{\lambda}  \right)} + e^{-\tilde{\lambda} + \frac{\tilde{\lambda}}{e}}\sum\limits_{j = \ceil*{\Phi_n}}^{\infty} \frac{ e^{-\tilde{\lambda} /e} \left(\tilde{\lambda} / e\right)^j}{j!} \nonumber\\
 \leq & 1 - e^{-(\Phi_n+1) \log \left(e(\Phi_n + 1) / \tilde{\lambda}  \right)} + e^{-\tilde{\lambda} + \frac{\tilde{\lambda}}{e}} \cdot \myprob[Q \geq \Phi_n] \nonumber\\
  < & 1 - e^{-(\Phi_n+1) \log \left(e(\Phi_n + 1) / \tilde{\lambda}\right)} + e^{-(1-1/e)\tilde{\lambda}} \cdot e^{-\Phi_n \log \left(\Phi_n /\tilde{\lambda}\right)} \label{eqn:deletion_only_quite_sparse_lower_bound_law_of_total_2}
\end{align}
where Eqn. (\ref{eqn:deletion_only_quite_sparse_lower_bound_law_of_total}) and (\ref{eqn:deletion_only_quite_sparse_lower_bound_law_of_total_2}) hold due to Lemma \ref{lem:binomial_upperbound} (note that $\Phi \gg \tilde{\lambda} / e$) and Lemma \ref{lem:ER_lowerbound_well-known}. Routine calculations show that for $n$ large enough, we have %\yusu{Minghao: I am not sure how to get the following two inequalities. }
\begin{align*}
(\Phi_n+1) \log \left(e(\Phi_n + 1) / \tilde{\lambda}\right) &\leq \frac{1}{2}\log n + 1\\
\Phi_n \log \left(\Phi_n / \tilde{\lambda}\right) &\geq \frac{1}{2}\log n - 1
\end{align*}
and $e^{-(1-1/e)\tilde{\lambda}} \leq 1 / (2e^2)$ since $nr^d \gg n^{-\epsilon}$ for all $\epsilon > 0$. Thus
\begin{align*}
\myprob \left[ \omega\left(G\left(\mathcal{N}\left(\mathbf{X}_i\right), 1-\pp\right) \right) \leq k_n\right] < 1 - \frac{1}{2e}n^{-1/2}.
\end{align*}
Plugging this back into Eqn. (\ref{eqn:deletion_only_large_product}), we have
\begin{align*}
\myprob \left[ \omega\left(G_N^{\pp,0} \right) \leq k_n \right] < \left(1 - \frac{1}{2e}n^{-1/2}\right)^m \leq e^{-\frac{1}{2e}n^{-1/2}m}
\end{align*}
Recall $m = \Omega(r^{-d})$ and $nr^d \ll \log n$, thus $n^{-1/2}m = \Omega\left(\sqrt{n} /\log n \right)$. This implies $\myprob \left[ \omega\left(G_N^{\pp,0} \right) \leq k_n \right] = o(1)$. 
Since we have that $\myprob \left[ \omega\left(G_n^{\pp,0} \right) \leq k_n \right] = o(1)$ with $$k_n = \Omega\left( \log \frac{\log n}{\log \left(\log n / nr^d\right)}\right),$$ we thus obtain the lower bound in Part (\rom{2}) of Theorem \ref{thm:deletion_only_RGG}.

%%%%%%%%%%%%%%%%%%%%%%%%%%%%%%%%%%%%%%%%%%%%%%%%%%%%%%%%%%%%%%%%%%%%%%%%%
\section{Combined case}\label{sec:combined_case}
In this section, we focus on bounding the clique number $\omega\left(G_n^{\pp,\qqq} \right)$ of $G_n^{\pp,\qqq}$, for different regimes of $nr^d$, $\pp$ and $\qqq$. Analogously to the monotonicity of the clique number of \myER{} random graphs \cite{janson2011random}, we have the following two monotone properties: for any positive integer $K$,
\begin{align*}
 \myprob\left[\omega\left(G_n^{0,\qqq}\right) \leq K\right]
 \leq \myprob\left[\omega\left(G_n^{\pp,\qqq}\right) \leq K\right] \leq
 \myprob\left[\omega\left(G_n^{\pp,0}\right) \leq K\right]\\
\myprob\left[\omega\left(G_n^{\pp,0}\right) \geq K\right] \leq \myprob\left[\omega\left(G_n^{\pp,\qqq}\right) \geq K\right] \leq \myprob\left[\omega\left(G_n^{0,\qqq}\right) \geq K\right]
\end{align*}
Combining these properties with Theorem \ref{thm:insertion_only_RGG}, Theorem \ref{thm:deletion_only_RGG} and technical lemmas (Lemma \ref{lem:ER_lowerbound_well-known} and Corollary \ref{thm:q_insertion_lowerbound}), we can derive some of the results showing below (part (\rom{1}) and part (\rom{3}.b)). 
%\yusu{Minghao: I don't think you need Corollary \ref{thm:q_insertion_lowerbound}. Don't you need Lemma \ref{lem:ER_lowerbound_well-known} combined with  that $\myprob[\omega(G_n^{q,\qqq}) \ge K] \ge \myprob[\omega(G(n,\qqq)) \ge K] $ for any $q$? At least for second bullet of part (I) I think that is what you need. Please double check.} \minghao{You can only get an upper bound for constant $p$ by applying Lemma \ref{lem:ER_lowerbound_well-known} (note the condition there), but here $p$ can be o(1). It's a little bit tricky I admit. }
Other results can be derived by carefully choosing $\aK_n$ in the proof of Theorem \ref{thm:type_2_quite_sparse} and Theorem \ref{thm:type_2_dense} to fit the corresponding lower bound for different regimes of $nr^d$. For example, we can set some $$\aK_n = \Theta\left(\log \left(\frac{\log n}{\log \left(\log n /nr^d \right)} \right)\right)$$ (the lower bound in the subcritical regime for deletion-only case; see Theorem \ref{thm:deletion_only_RGG}) in the proof of Theorem \ref{thm:type_2_quite_sparse} to derive part (\rom{2}.a) of Theorem \ref{thm:combined_RGG}. 
For these reasons,
we omit the proof of the following theorem.

\begin{theorem}\label{thm:combined_RGG}
Given a $(\pp, \qqq)$-perturbed noisy random geometric graph $G_n^{\pp,\qqq}$ in \mysetting{} with a fixed constant $0 < \pp < 1$, the following holds: 

\begin{enumerate}\denselist
\item[(\rom{1})] Suppose that $nr^d \leq n ^{-\alpha}$ for some fixed $\alpha \in \left( 0, 1 / \myBC^2\right]$. Then there exist constants $C_1,C_2$ such that
\begin{itemize}\denselist
\item[(\rom{1}.a)] if $\qqq \leq \left(1 / n\right)^{C_1}$, then a.s.
\begin{align*}
\omega\left(G_n^{\pp,\qqq}\right) \sim 1
\end{align*}
\item[(\rom{1}.b)] and if $\left(1 / n\right)^{C_1} < \qqq \leq C_2$, then a.s. 
\begin{align*}
\omega\left(G_n^{\pp,\qqq}\right) \sim \log_{ 1 / \qqq}{n} 
\end{align*} 
\end{itemize}

\item[(\rom{2})] Suppose that $n^{-\epsilon} \ll nr^d \ll \log{n}$ for all $\epsilon > 0$. Then there exist constants $C_1, C_2, C_3$ such that
\begin{itemize}\denselist
\item[(\rom{2}.a)] if $$\qqq \leq \left(1 / n\right)^{C_1 / \log \frac{\log{n}}{\log \left(\log{n} /nr^d\right)}},$$ then a.s.
\begin{align*}
\omega\left(G_n^{\pp,\qqq}\right) \sim \log\dfrac{\log{n}}{\log{\left(\log{n}/nr^d \right)}} 
\end{align*}
%\item $\exists$ a constant $0 < C_3 < 1$ such that if $\left(\frac{nr^d}{\ln{n}}\right)^{\xi} \ll q \leq C_3 < 1$ for all $\xi > 0$, then
\item[(\rom{2}.b)] and if $\left( nr^d / \log n\right)^{C_2} < \qqq \leq C_3$, then a.s. 
\begin{align*}
\omega\left(G_n^{\pp,\qqq}\right) \sim \log_{1 / \qqq}{n}
\end{align*} 

%\item otherwise
%\begin{align*}
%\omega\left(G_n^{\pp,\qqq}\right) & = O\left(\frac{\ln n}{\ln \frac{\ln n}{nr^d}}\right) ~~~~\text{a.s.}\\
%\omega\left(G_n^{\pp,\qqq}\right) & = \Omega\left( \ln\dfrac{\ln{n}}{\ln{\frac{\ln{n}}{nr^d}}} \right) ~~~~\text{a.s.}
%\end{align*} 
\end{itemize}

\item[(\rom{3})] There exists a constant $T > 0$ such that if $\sigma n r^d / \log n \rightarrow t  \in (T, \infty)$, then there exist constant $C_1,C_2$ such that
\begin{itemize}
    \item[(\rom{3}.a)] if $\qqq \leq \left(1 / n\right)^{C_1 / \log \log n} \left(\log \log n / \log n\right)$, then a.s.
    \begin{align*}
\omega\left(G_n^{\pp,\qqq}\right) \sim \log{\left(nr^d\right)} 
\end{align*}
    %\item if $\qqq > \left(\frac{1}{n}\right)^{\frac{C_1}{\ln \ln n}} \frac{\ln \ln n}{\ln n}$ and $\qqq = o(1)$, then
    %\begin{align*}
%\omega\left(G_n^{\pp,\qqq}\right) & = O\left(nr^d \right)  ~~~~\text{a.s.}\\
%\omega\left(G_n^{\pp,\qqq}\right) & = \Omega\left( \ln{\left(nr^d\right)} \right)  ~~~~\text{a.s.}
%\end{align*}
    
\item[(\rom{3}.b)] and if $0 < \qqq \leq C_2$ and $\qqq = \Theta(1)$, then a.s.
\begin{align*}
\omega\left(G_n^{\pp,\qqq}\right) \sim \log_{1 / \qqq}{n} 
\end{align*}
\end{itemize}

\end{enumerate}
\end{theorem}

%%%%%%%%%%%%%%%%%%%%%%%%%%%%%%%%%%%%%%%%%%%%%%%%%%%%%%%%%%%%%%%%%%%%%%%%%%%
\section{Concluding remarks}

%This work can also be easily extended beyond Euclidean space, say a geodesic space equipped with doubling measure \cite{kahle2018local}. The different behaviors of edge clique numbers (local cliques) are discussed there.

In this paper, we study the behavior of the clique number of noisy random geometric graphs $G_n^{\pp, \qqq}$. In particular, we give the asymptotic tight bounds for the insertion-only case $G_n^{0, \qqq}$ and the deletion-only case $G_n^{\pp, 0}$ under different assumptions on $nr^d$ (Theorem \ref{thm:insertion_only_RGG} and Theorem \ref{thm:deletion_only_RGG}, respectively). To obtain these results, we deploy a range of classical and new techniques: For example, we develop a novel approach based on what we call the "well-separated clique-partitions family" to handle the insertion case. 
Some partial results for the general case  $\omega\left(G_n^{\pp, \qqq}\right)$ are also provided (Theorem \ref{thm:combined_RGG}). 
We also note that results in our paper can be extended beyond the Euclidean setting: For example, in \cite{kahle2018local}, noisy random geometric graphs generated from points sampled from a well-behaved doubling measure supported on a geodesic space are considered, and behaviors of the edge clique number are investigated. 

This work represents a first step towards characterizing properties of the noisy random geometric graphs (which intuitively are generated based on two types of random processes). There are many interesting open problems. For example, the combined case is not yet completely resolved (there are still gaps in the regimes). Also currently we only provide asymptotic tight bounds, and it would be interesting to identify the exact constant for the high order terms too. 
It will also be interesting to study other quantities beyond the clique number. 

Finally, we note that the random deletions/insertions can be viewed as ``noise'' on top of a base graph (which is a random geometric graph in our work). It will be interesting to see whether studies of clique numbers of other quantities can be used to ``denoise'' the input graph in practical applications (e.g, as in  \cite{ParthasarathyST17}). 
Indeed, the work of \cite{ParthasarathyST17} showed that the shortest path metric of the random geometric graph can be recovered (with approximation guarantees) from its ER-perturbed version, if the insertion probability $\qqq$ is small (compared to expected degree). In particular, in the high level, the work of \cite{ParthasarathyST17} uses a quantity, called the Jaccard index, to identify what they refer to as ``long-edges", removes such edges and use the shortest distances in the denoised graph to approximate those of the underlying random geometric graphs with approximation guarantees. We believe that the edge clique number that we study in this paper can be used as a more powerful way to identify such ``long-edges" that can tolerate a much larger range of insertion probability $\qqq$ than the previous work in Jaccard index. We will leave this as an interesting direction to explore in the future.

\bibliography{clique_number_RGG_under_ER_perturbation}

\appendix

\section{The missing proofs in Section \ref{sec:proof_insertion_only_RGG}}
\label{appendix:thm:insertiononly_very_sparse}

\subsection{The proof of Lemma \ref{lem:insertioncaseA_very_sparse}}
\label{appendix:lem:insertioncaseA_very_sparse}
%\begin{proof}
Since $1 \leq k \leq N_{\max}$, we know that $\qqq \leq 1$ thus is well-defined as a probability. To estimate the summation on the right hand side of Eqn. (\ref{eqn:expectationX_very_sparse}), we consider the quantity $x_{\max} := \max_{i}\{x_i\}$. We first enumerate all the possible cases of $(x_1, x_2, \cdots, x_m)$ when $x_{\max}$ is fixed, and then vary the value of $x_{\max}$.  

Set $h(y) = \max_{x_{\max}=y}\left\{\sum_{i=1}^{m}x_i^2\right\}$ for $y \geq \ceil*{k /m}$. It is the maximum value of $\sum_{i=1}^{m}x_i^2$ under the constraint $x_{\max}=y$. Without loss of generality, we assume $x_1=y$ and $y \geq x_2 \geq x_3 \geq \cdots \geq x_m \geq 0$. We argue that $\argmax_{x_{\max}=y}\left\{\sum_{i=1}^{m}x_i^2\right\} = \{y, y, \cdots, y, k-ry, 0, \cdots, 0\}$, that is $x_1 = x_2 = \cdots = x_r = y, x_{r+1}=k-ry$ where $r = \floor*{k /y}$. 

To show this, we first consider $x_2$: if $x_2 = y$, then consider $x_3$; otherwise, $x_2 < y$, then we search for the largest index $j$ such that $x_j >0$. Note the fact that if $x \geq y>0$, then $(x+1)^2+(y-1)^2 = x^2 + y^2 + 2(x-y) + 2>x^2+y^2$. So if we increase $x_2$ by $1$ and decrease $x_j$ by $1$, we will enlarge $\sum_{i=1}^{m}x_i^2$. After we update $x_2 = x_2 + 1$, $x_j = x_j - 1$, we still get a decreasing sequence $x_1\geq x_2 \geq \cdots \geq x_m \geq 0$. If we still have $x_2 < y$, 
then we repeat the same procedure above (by increasing $x_2$ and decreasing $x_j$ where $j$ is the largest index such that $x_j>0$). 
We repeat this process until $x_2 = y$ or $x_1 + x_2 = k$. If it is the former case (i.e, $x_2=y$), then we consider $x_3$ and so on. Finally, we will get the sequence $x_1 = x_2 = \cdots = x_r = y, x_{r+1}=k-ry$ where $r = \floor*{k /y}$ as claimed, and this setting maximizes $\sum_{i=1}^{m}x_i^2$. 

Next we claim that $h(y+1) > h(y)$. The reason is similar to the above. We update the sequence $x_1 = x_2 = \cdots = x_r = y, x_{r+1}=k-ry$ (which corresponding to $h(y)$) from $x_1$: we increase $x_1$ by $1$; search the largest index $s$ such that $x_{s} > 0$ and decrease $x_{s}$ by $1$. And then consider $x_2$ and so on and so forth. This process won't stop until $x_1 = x_2 = \cdots = x_s =y+1$ and $x_{s+1} = k - s(y+1)$ with $s = \floor*{k / y+1}$. Thus $h(y+1)>h(y)$. 

By enumerating all the possible values of $x_{\max}$, we split Eqn. (\ref{eqn:expectationX_very_sparse}) into three parts as follows (corresponding to the cases when $x_{\max} = k, x_{\max} \in \left[ \ceil*{ (k+1) / 2}, k-1\right] \text{~and~}x_{\max} \in \left[\ceil*{k / m}, \ceil*{(k+1)/2}-1\right]$) (see the remarks after this equation for how the inequality is derived); 
\begin{align}\label{eqn:mainestimate_very_sparse}
&\qqq^{2k}\sum\limits_{\substack{x_1 + x_2 + \cdots + x_m = k \\x_i \geq 0}} \binom{N_{\max}}{x_1} \binom{N_{\max}}{x_2}\cdots \binom{N_{\max}}{x_m}\qqq^{(k^2-\sum_{i=1}^{m}x_i^2)/2} \nonumber\\
\leq &~ \qqq^{2k}\binom{N_{\max}}{k}m + \qqq^{2k}\sum\limits_{x_{\max} =\ceil*{\frac{k+1}{2}}}^{k-1}\Bigg(\binom{m}{1} \binom{N_{\max}}{x_{\max}} \sum\limits_{\substack{\sum_{i=1}^{m-1}y_i  = k - x_{\max}\\ 0 \leq y_i \leq x_{\max}}}\binom{N_{\max}}{y_1} \cdots \\
&~~~~ \binom{N_{\max}}{y_{m-1}}\qqq^{x_{\max}(k-x_{\max})} \Bigg) + \binom{mN_{\max}}{k}\qqq^{\frac{(k-1)^2}{4} + 2k}\nonumber
\end{align}

The first term on the right hand side of Eqn. (\ref{eqn:mainestimate_very_sparse}) comes from the fact that if $x_{\max} = k$, then all the possible cases for $(x_1, x_2, \cdots, x_m)$ are $(k,0,0,\cdots,0), \cdots, (0,\cdots,0,k)$, and there are $m$ cases all together. 
For each case, the value of each term in the summation is $\binom{N_{\max}}{k}$, giving rise to the first term in Eqn. (\ref{eqn:mainestimate_very_sparse}).

The third term on the right hand side of Eqn. (\ref{eqn:mainestimate_very_sparse}) can be derived as follows. First, observe that 
\begin{align*}
&\sum\limits_{x_{\max} = \ceil*{\frac{k}{m}}}^{\ceil*{\frac{k+1}{2}} - 1}\Bigg(\sum\limits_{\substack{x_1 + x_2 + \cdots + x_m = k \\x_i \geq 0, \max_{i}\{x_i\} = x_{\max}}} \binom{N_{\max}}{x_1} \binom{N_{\max}}{x_2}\cdots \binom{N_{\max}}{x_m}\Bigg)\\
\leq &\sum\limits_{\substack{x_1 + x_2 + \cdots + x_m = k \\x_i \geq 0}} \binom{N_{\max}}{x_1} \binom{N_{\max}}{x_2}\cdots \binom{N_{\max}}{x_m} = \binom{mN_{\max}}{k} . 
\end{align*}
On the other hand, as $x_{\max} \leq \ceil*{(k+1)/2} - 1 = \ceil*{(k-1)/2}$, we have:
\begin{align*}
\frac{k^2-\sum_{i=1}^{m}x_i^2}{2} \geq \frac{k^2 -h(x_{\max})}{2} \geq \frac{k^2 -h(\ceil*{\frac{k-1}{2}})}{2} \geq \frac{(k-1)^2}{4}, 
\end{align*}
where the second inequality uses the fact that $h(y)$ is an increasing function, and the last inequality comes from that $h(\ceil*{(k-1)/2}) \le (\ceil*{(k-1)/2})^2 + (\ceil*{(k-1)/2})^2 + 1 \le k^2/4 + k^2/4 + 1 = k^2/2 + 1$. 
\vspace{5mm}

In what remains, it suffices to estimate all the three terms on the right hand side of Eqn. (\ref{eqn:mainestimate_very_sparse}). We will repeatedly use the well-known combinatorial inequality $\binom{n}{k} < \left(en / k\right)^k$.

\myparagraph{\bf The first term of Eqn. (\ref{eqn:mainestimate_very_sparse}):~} According to the assumptions in Eqn. (\ref{eqn:qbound1_very_sparse}), we know $$\qqq \leq \frac{1}{\sqrt{e}} \left(\frac{1}{n^3m}\right)^{\frac{1}{2k}} \left( \frac{k}{N_{\max}}\right)^{\frac{1}{2}}.$$ Thus, for the first term of Eqn. (\ref{eqn:mainestimate_very_sparse}), we have:
\begin{align}\label{eqn:firsttermestimate_very_sparse}
\qqq^{2k}\binom{N_{\max}}{k}m ~~<~~ \left(\frac{1}{e^k} \left(\frac{1}{n^3m}\right) \left( \frac{k}{N_{\max}}\right)^{k}\right) \left(\dfrac{eN_{\max}}{k}\right)^{k}m ~~=~~ \frac{1}{n^{3}}
\end{align}

\paragraph{\bf The second term of Eqn. (\ref{eqn:mainestimate_very_sparse}):~}For the second term of Eqn. (\ref{eqn:mainestimate_very_sparse}), we relax the constraint $x_{\max} \geq y_i \geq 0$ to $y_i \geq 0$. Thus, we have:
\begin{align}\label{eqn:boxandballtrick_very_sparse}
\sum\limits_{\substack{y_1+\cdots+y_{m-1} = k - x_{\max}\\x_{\max} ~~\geq~~ y_i \geq 0}}\binom{N_{\max}}{y_1}\cdots &\binom{N_{\max}}{y_{m-1}} \leq \sum\limits_{\substack{y_1+\cdots+y_{m-1} = k - x_{\max}\\y_i \geq 0}}\binom{N_{\max}}{y_1}\cdots \binom{N_{\max}}{y_{m-1}}\nonumber\\
& =~~ \binom{(m-1)N_{\max}}{k-x_{\max}} ~~<~~ \left( \frac{e(m-1)N_{\max}}{k-x_{\max}} \right)^{k-x_{\max}} 
\end{align}

Now apply (\ref{eqn:boxandballtrick_very_sparse}) to the second term of (\ref{eqn:mainestimate_very_sparse}), we have (starting from the second line, we replace $x_{max}$ to be $j$ for simplicity):
\begin{align}
\qqq^{2k}\sum\limits_{x_{\max} =\ceil*{\frac{k+1}{2}}}^{k-1}&\Bigg(\binom{m}{1} \binom{N_{\max}}{x_{\max}} \sum\limits_{\substack{\sum_{i=1}^{m-1}y_i = k - x_{\max}\\0 \leq y_i \leq x_{\max}}}\binom{N_{\max}}{y_1} \cdots \binom{N_{\max}}{y_{m-1}}\qqq^{x_{\max}(k-x_{\max})} \Bigg) \nonumber\\
<& \sum\limits_{j =\ceil*{\frac{k+1}{2}}}^{k-1}\Bigg(m \left(\frac{eN_{\max}}{j}\right)^{j}\qqq^{2k + j(k-j)} \left( \frac{e(m-1)N_{\max}}{k-x_{\max}} \right)^{k-x_{\max}}  \Bigg)\nonumber\\
 =& \sum\limits_{j =\ceil*{\frac{k+1}{2}}}^{k-1}\Bigg(m^{k-j+1} N_{\max}^{k} e^k \left(\frac{1}{j}\right)^j \left(\frac{1}{k-j} \right)^{k-j} \qqq^{2k + j(k-j)} \Bigg)\nonumber\\
 <& \sum\limits_{j =\ceil*{\frac{k+1}{2}}}^{k-1}\Bigg(m^{k-j+1} N_{\max}^{k} e^k \left(\frac{2}{k}\right)^{k} \qqq^{2k + j(k-j)} \Bigg)\label{eqn:secondtermmain}
\end{align}
where the last inequality holds due to the inequality of arithmetic and geometric means.
%\yusu{Minghao, there are a few places below and also in Appendix \ref{appendix:lem:insertioncaseB_very_sparse}, you use $n^{\epsilon}$ instead of $n^3$, and I changed them. But later when you read, be careful in case there are some still remain.}

Note that by tedious elementary calculation, we know $[2k+j(k-j)] /k\geq (k-j+1) /2 \geq 1$ when $\ceil*{(k+1)/2} \leq j \leq k-1$.
Since $$\qqq \leq \frac{1}{2ek^{\frac{1}{k}}}\left(\frac{1}{n^3m^2} \right)^{\frac{1}{k}} \frac{k}{N_{\max}} < 1$$ by Eqn. (\ref{eqn:qbound1_very_sparse}), for each $j$ satisfying $\ceil*{(k+1)/2} \leq j \leq k-1$, we have:
\begin{align}
&m^{k-j+1} N_{\max}^{k} e^k \left(\frac{2}{k}\right)^{k} \qqq^{2k + j(k-j)} \nonumber\\
\leq & m^{k-j+1} N_{\max}^{k} e^k\left(\frac{2}{k}\right)^{k} \qqq^{k\left(\frac{k-j+1}{2}\right)}\nonumber\\ 
\leq & m^{k-j+1} N_{\max}^{k} e^k\left(\frac{2}{k}\right)^{k} \left(\frac{1}{\left( 2e \right)^k} \frac{1}{kn^3} \frac{k^k}{N_{\max}^k} \right)^{\frac{k-j+1}{2}} \left(\frac{1}{m}\right)^{k-j+1} \nonumber\\ 
\leq & N_{\max}^{k} e^k\left(\frac{2}{k}\right)^{k} \left(\frac{1}{\left( 2e \right)^k} \frac{1}{kn^3} \frac{k^k}{N_{\max}^k} \right) \nonumber\\
= &\frac{1}{kn^{3}}\label{eqn:uselowerbound_very_sparse}
\end{align}
where the inequality on the fourth line holds as $k \leq N_{\max}$ and $(k-j+1)/2 \geq 1$.

\paragraph{\bf The third term of Eqn. (\ref{eqn:mainestimate_very_sparse}):~} For the third term of (\ref{eqn:mainestimate_very_sparse}), note that $(k-1)^2 / 4 + 2k >  k^2 / 4$ and by plugging in the condition $$\qqq \leq \frac{1}{e^{\frac{4}{k}}}\left(\frac{1}{n^3m^k} \right)^{\frac{4}{k^2}}\left(\frac{k}{N_{\max}} \right)^{\frac{4}{k}}<1,$$ we have
\begin{align}\label{eqn:2ndterm_very_sparse}
\binom{mN_{\max}}{k} \qqq^{\frac{(k-1)^2}{4}+2k} < \left(\frac{emN_{\max}}{k}\right)^{k}\qqq^{\frac{k^2}{4}} \leq \left(\frac{emN_{\max}}{k}\right)^{k} \frac{1}{e^k} \frac{1}{n^3m^k}\frac{k^k}{N_{\max}^k} = \frac{1}{n^{3}}
\end{align}

Finally, combining (\ref{eqn:firsttermestimate_very_sparse}), (\ref{eqn:uselowerbound_very_sparse}) and (\ref{eqn:2ndterm_very_sparse}), we have:
\begin{align*}
E[\aI ] \leq \frac{1}{n^{3}} + \frac{k}{2} \cdot \frac{1}{kn^{3}}  + \frac{1}{n^{3}} = \frac{5}{2n^{3}}
\end{align*}
This concludes Lemma \ref{lem:insertioncaseA_very_sparse}. 

\subsection{The missing details in case (\romannumeral 2) of part (a) of Theorem \ref{thm:type_2_very_sparse}}\label{appendix:detailscaseBinsertiononly_very_sparse}
Set $N_u := |B^{\mathcal{X}_n}_r(u)|$ and $N_v := |B^{\mathcal{X}_n}_r(v)|$. Let $\tilde{k} := \floor*{\aK / 2} - 2$. Easy to see $\tilde{k} \geq 1$ since $\aK \geq 8 \myBC^2 \geq 8$. For every set $S$ of $(\tilde{k}+2)$ vertices in $\tilde{G}_{uv}^{local}$, let $A_S$ be the event ``\emph{$S$ is a clique in $\tilde{G}_{uv}^{local}$ containing edge $(u,v)$ given $\aF$}'' and $\aY_S$ its indicator random variable. Set
\begin{align*}
\aY = \sum\limits_{|S|=\tilde{k}+2} \aY_S. 
\end{align*}
Then $\aY$ is the number of cliques of size $(\tilde{k}+2)$ in $\tilde{G}_{uv}^{local}$ containing edge $(u,v)$ given $\aF$. Linearity of expectation gives:
\begin{align}\label{eqn:expectationI_very_sparse}
\myE[\aY] = \sum\limits_{|S|=\tilde{k}+2}\myE[\aY_S] = \sum\limits_{\substack{x_1 + x_2 = \tilde{k} \\ 0 \leq x_1 \leq N_u - 1 \\ 0\le x_2 \le N_v - 1}} \binom{N_u-1}{x_1} \binom{N_v-1}{x_2}\qqq^{(x_1+1)(x_2+1)-1}
\end{align}

To estimate this quantity, we first prove the following result:

\begin{lemma}\label{lem:insertioncaseB_very_sparse}
If $\tilde{k} \geq 1$ and
\begin{align}\label{eqn:qbound2_very_sparse}
\qqq ~~\leq~~ \frac{1}{e}\left(\frac{1}{n^3}\right)^{\frac{1}{\tilde{k}}}\frac{\tilde{k}}{N_u + N_v} 
\end{align}
hold, then $\myE[\aY] = O\left(n^{-3}\right)$
\end{lemma}

\begin{proof}
Easy to see that if $\tilde{k} > N_u + N_v - 2$, then the summation on the right hand side of Eqn. (\ref{eqn:expectationI_very_sparse}) is $0$. Now we move on to the case when $\tilde{k} \leq (N_u -1 )+ (N_v - 1) < N_u + N_v$. Easy to see $\qqq < 1$ in this case. Thus, the right hand side of (\ref{eqn:expectationI_very_sparse}) can be bounded from above by:

\begin{align}
&\sum\limits_{\substack{x_1 + x_2 = \tilde{k} \\ 0 \leq x_1 \leq N_u - 1 \\ 0 \leq x_2 \leq N_v - 1}} \binom{N_u-1}{x_1} \binom{N_v-1}{x_2}\qqq^{(x_1+1)(x_2+1)-1} \nonumber\\
&\leq~~ \sum\limits_{i=0}^{\tilde{k}}\binom{N_u}{i}\binom{N_v}{\tilde{k}-i}\qqq^{\tilde{k} + i(\tilde{k}-i)} ~~<~~ \sum\limits_{i=0}^{\tilde{k}}\binom{N_u}{i}\binom{N_v}{\tilde{k}-i}\qqq^{\tilde{k}} ~~=~~ \binom{N_u+N_v}{\tilde{k}}\qqq^{\tilde{k}} \nonumber\\
&~~<~~ \left(\frac{e(N_u + N_v)}{\tilde{k}}\right)^{\tilde{k}}\qqq^{\tilde{k}} ~~\leq~~ \frac{1}{n^3} \nonumber
\end{align}
where the last inequality holds due to condition (\ref{eqn:qbound2_very_sparse}).
\end{proof}

%The proof of this technical result can be found in Appendix \ref{appendix:lem:insertioncaseB_very_sparse}. 

Easy to see that there exist two constants $c^{b}_1$ and $c^{b}_2$ which depend on the \Bconst{} $\myBC$ and $\alpha$, 
such that %if $\aK \leq 4 N_{\max}$ and 
\begin{align*}
\text{if}  ~~ \aK \geq 8 \myBC^2 ~~\text{and}~~&\qqq \leq  c^b_1 \cdot \left(1 / n\right)^{c^b_2/\aK}, \text{ then the conditions in Eqn. (\ref{eqn:qbound2_very_sparse}) will hold. }
\end{align*}

On the other hand, we have
\begin{align*}
\myprob\left[ \omega_{u,v}\left(\tilde{G}_{uv}^{local}\right) \ge \aK / 2 \middle| \aF \right] = \myprob[Y > 0] \leq \myE[Y]
\end{align*}
Thus, by Lemma \ref{lem:insertioncaseB_very_sparse}, we know that 
\begin{align}\label{eqn:caseBalmostfinal_very_sparse}
\text{If} ~\aK \geq 8 \myBC^2 &~\text{and}~ \qqq \leq ~ c^b_1 \cdot \left(1 / n\right)^{c^b_2/\aK}, ~\text{then} ~ \myprob \left[\omega_{u,v}\left(\tilde{G}_{uv}^{local}\right) \ge \aK / 2 \middle| \aF \right] = O(n^{-3})
\end{align}

\subsection{The proof of Lemma \ref{lem:insertioncaseA_very_sparse_looser}}
\label{appendix:lem:insertioncaseA_very_sparse_looser}
\begin{proof}
By using a similar argument as in Appendix \ref{appendix:lem:insertioncaseA_very_sparse}, it is easy to see that the maximum value of $\sum_{i=1}^{m}x_i^2$, under the constraints $\sum_{i=1}^{m}x_i = k$ and $x_i \in [0, N_{\max}]$ for each $i\in [1,m]$, is $rN_{\max}^2 + (k - rN_{\max})^2$ where $r = \floor*{k /N_{\max}}$. The maximum can be achieved when $(x_1, x_2, \cdots, x_m) = (\underbrace{N_{\max}, \cdots, N_{\max}}_{r}, k - rN_{\max}, \ldots)$. Thus, we have

\begin{align}
\myE[\aI] \leq &~~ \qqq^{2k}\sum\limits_{\substack{x_1 + x_2 + \cdots + x_m = k \\x_i \geq 0}} \binom{N_{\max}}{x_1} \binom{N_{\max}}{x_2}\cdots \binom{N_{\max}}{x_m}\qqq^{(k^2-\sum_{i=1}^{m}x_i^2)/2} \nonumber\\
< & \left(\sum\limits_{\substack{x_1 + x_2 + \cdots + x_m = k \\x_i \geq 0}} \binom{N_{\max}}{x_1} \binom{N_{\max}}{x_2}\cdots \binom{N_{\max}}{x_m}\right) \cdot \qqq^{\frac{k^2 - \left(rN_{\max}^2 + (k - rN_{\max})^2\right)}{2} + 2k}\nonumber\\
= & \binom{mN_{\max}}{k} \qqq^{\frac{k^2 - \left(rN_{\max}^2 + (k - rN_{\max})^2\right)}{2} + 2k}\nonumber\\
< & \left(\frac{emN_{\max}}{k}\right)^{k} \qqq^{k(rN_{\max} + 1) - \frac{r(r+1)}{2}N_{\max}^2}\nonumber\\
< & \left(\frac{emN_{\max}}{k}\right)^{k} \qqq^{k\left(\left(\frac{k}{N_{\max}} -1\right)N_{\max} + 1\right) - \frac{1}{2} \frac{k}{N_{\max}} \left(\frac{k}{N_{\max}} +1 \right)  N_{\max}^2}\label{eqn:mainestimate_very_sparse_looser_1}\\
= & \left(\frac{emN_{\max}}{k}\qqq^{\frac{1}{2}k - \frac{3}{2}N_{\max} + 1} \right)^k\label{eqn:mainestimate_very_sparse_looser}
\end{align}
where Eqn. (\ref{eqn:mainestimate_very_sparse_looser_1}) holds since $k / N_{\max} \geq r > k / N_{\max} - 1$. 

Pick a constant $C_3$, which only depends on the \Bconst{} $\myBC$ and $\alpha$, such that $$\frac{C_3 \floor*{\log_{1 /\qqq}n}}{2 |\Lambda|} - 3 \geq 16 \log_{1 / \qqq}n \geq 6 N_{\max}.$$ This can be done since $N_{\max}$ is a constant and $\left(1 /n \right)^{8 / \left(3N_{\max}\right)} \leq \qqq  < 1$ implies $\log_{1 / \qqq}n \geq  3N_{\max}/8$. Set $\aK = C_3 \floor*{\log_{1 /\qqq}n}$. Recall that $k = \floor*{\aK / (2 |\Lambda|)} - 2$, thus $k \geq 16 \log_{1 / \qqq}n \geq 6 N_{\max}$. Also note that $m \leq n$. Hence, we have the following inequality.
%\yusu{Minghao: double check the second inequality: it seems to be missing a `1'.}
\begin{align}\label{eqn:mainestimate_very_sparse_looser_continue}
 \left(\frac{emN_{\max}}{k}\qqq^{\frac{1}{2}k - \frac{3}{2}N_{\max} + 1} \right)^{k} <  \left(\frac{enN_{\max}}{k}\qqq^{\frac{1}{4}k} \right)^{k} \leq \left(\frac{eN_{\max}}{k}n^{-3} \right)^{k} = O(n^{-3})
\end{align}

Finally, combining (\ref{eqn:mainestimate_very_sparse_looser}) and (\ref{eqn:mainestimate_very_sparse_looser_continue}), we have $\myE[\aI] = O(n^{-3})$.
\end{proof}

%%%%%%%%%%%%%%%%%%%%%%%%%%%%%%%%%%%%%%%%%%%%%%%%%%%%%%%%%
\subsection{Proof of Lemma \ref{thm:ER_lowerbound}}
\label{appendix:thm:ER_lowerbound}

We will use the standard second moment method to prove this lemma. For completeness, we first state the second moment method. For those who are familiar with this method, our main technical step is to estimate the summation on the right hand side of Eqn. (\ref{eqn:2nd_delta*/mu}).
%We state the second moment method for completeness. For those who are familiar with this method, the key step is to estimate the summation on the right hand side of Eqn. (\ref{eqn:2nd_delta*/mu}).

\begin{definition}[Symmetric random variables]
We say random variables $Z_1, \cdots, Z_m$, where $Z_i$ is the indicator random variable for event $U_i$, are \textbf{symmetric} if for every $i \neq j$ there is a measure preserving mapping of the underlying probability space that permutes the $m$ events and sends event $U_i$ to event $U_j$.
\end{definition}

Let $Z$ be a nonnegative integral-valued random variable, and suppose we have a decomposition $Z = Z_1 + \cdots + Z_m$, where $Z_i$ is the indicator random variable for event $U_i$ and $Z_1, \cdots, Z_m$ are symmetric. For indices $i,j$, write $i \sim j$ if $i \neq j$ and the events $U_i, U_j$ are not independent. For any fixed index $i$, we set 
\begin{align*}
\Delta^* := \sum\limits_{j \sim i}\myProb[U_j \mid U_i], 
\end{align*}
and note that by the symmetry of $Z_i$, $\Delta^*$ is independent of the index $i$ (thus we are not denoting it by $\Delta_i^*$). 
%where $i$ is any fixed index. (Note that by the symmetry of $Z_i$, $\Delta^*$ is independent of $i$.)

\begin{theorem}[The second moment method \cite{alon2016probabilistic}]\label{thm:2ndmoment}
If $\myE[Z] \rightarrow \infty$ and $\Delta^* = o(\myE[Z])$ as $m \rightarrow \infty$, then $\myProb[Z = 0] \rightarrow 0$.  
\end{theorem}

Now we are ready to prove Lemma \ref{thm:ER_lowerbound}.

\begin{proof}[Proof of Lemma \ref{thm:ER_lowerbound}]
Set $k = \floor*{\log_{1 / p}{n}}$. Now consider all the $k-$set $S_i$ of vertices in $G(n,p)$. Let $U_i$ be the event ``$S_i$ is a clique'' and $Z_i$ its indicator random variable. (All $k-$sets ``look the same'' so that the $Z_i$'s are symmetric.) $I$ is a finite index set enumerating all the $k-$sets in $G(n, p)$. Set
\begin{align*}
Z = \sum\limits_{i\in I}Z_i
\end{align*}
so that $Z$ is the number of $k-$cliques in $G(n,p)$. Linearity of Expectation gives:
\begin{align*}
\myE[Z] = \sum\limits_{i \in I}\myE[Z_i]= \binom{n}{k}p^{\binom{k}{2}}
\end{align*}
Easy to see that
\begin{align*}
\Delta^{*} =\sum\limits_{j \sim i}\myProb[U_j | U_i] = \sum\limits_{\ell=2}^{k-1}\binom{k}{\ell}p^{\binom{k}{2}-\binom{\ell}{2}}\binom{n-k}{k-\ell}
\end{align*}

Since $k = \floor*{\log_{1 / p}{n}}$ and $p \leq (1 / n)^{1 / \sqrt[4]{n}}$, we know that $p^{k-1} > p^k > 1 / n$, $k \leq n^{1/4}$ and $\log k / \log n \leq 1/4$. Also note that $p \geq \left(1 / n\right)^{1 /11}$. Easy to see that
\begin{align*}
k + 1 > \log_{1 / p}{n} \geq \frac{1}{\xi} = 11 
\end{align*}
which further implies $k > 10$.  

Note that for sufficiently large $n$, we have $n - k > n / 2$. Thus, using $p^{k-1} > 1 / n$ as derived earlier, we have: 
\begin{align*}
\myE[Z] = \binom{n}{k}p^{\binom{k}{2}} > \frac{(n-k)^k}{k^k}p^{\frac{k(k-1)}{2}} > \left(\frac{n}{2k}\right)^{k} n ^{-\frac{k}{2}} = n^{\frac{k}{2}\left(1-\frac{2\log{(2k)}}{\log n}\right)} > n^{\frac{k}{4}} \rightarrow \infty 
\end{align*}
To apply Theorem \ref{thm:2ndmoment}, it suffices to estimate the term $\Delta^* / \myE[Z]$.
\begin{align}\label{eqn:2nd_delta*/mu}
\frac{\Delta^*}{\myE[Z]} = \dfrac{\sum\limits_{\ell=2}^{k-1}\binom{k}{\ell}\binom{n-k}{k-\ell}(p)^{\binom{k}{2}-\binom{\ell}{2}}}{\binom{n}{k}(p)^{\binom{k}{2}}} =\sum\limits_{\ell=2}^{k-1}\dfrac{\binom{k}{\ell}\binom{n-k}{k-\ell}}{\binom{n}{k}}(p)^{-\binom{\ell}{2}}
\end{align}

We estimate the summation on the right hand side term by term. Let $$g(\ell) := \frac{\binom{k}{\ell}\binom{n-k}{k-\ell}}{\binom{n}{k}}(p)^{-\binom{\ell}{2}}.$$ 
%We want to show that $g(l) \leq \max\{g(2),g(k-1)\}$ and further get an upper bound for both $g(2)$ and $g(k-1)$. 
Note that for $\ell \in [2,k-1]$, we have
\begin{align*}
g(\ell) & ~~=~~ \frac{\binom{k}{\ell}\binom{n-k}{k-\ell}}{\binom{n}{k}}(p)^{-\binom{\ell}{2}} ~~\leq~~ \frac{\binom{k}{\ell}\frac{(n-k)^{k-\ell}}{(k-\ell)!}}{\frac{(n-k)^k}{k!}}n^{\frac{1}{k}\frac{\ell(\ell-1)}{2}} ~~\leq~~ \frac{\binom{k}{\ell}k^\ell}{(n-k)^\ell}n^{\frac{1}{k}\frac{\ell(\ell-1)}{2}} \nonumber\\
& ~~\leq~~ \frac{(\frac{ek}{\ell})^\ell k^\ell}{(n-k)^\ell}n^{\frac{1}{k}\frac{\ell(\ell-1)}{2}} ~~=~~ n^{-\frac{\ell\log{\left(\frac{\ell(n-k)}{ek^2}\right)}}{\log{n}} + \frac{1}{k}\frac{\ell(\ell-1)}{2}}
\end{align*}

Now set $$h(\ell) = -\frac{\ell\log{\left(\frac{\ell(n-k)}{ek^2}\right)}}{\log{n}} + \frac{1}{k}\frac{\ell(\ell-1)}{2};$$ and thus by the above inequality we have $g(\ell) \le n^{h(\ell)}$. 
We claim that $\forall \ell \in [2,k-1], h(\ell) \leq \max \{h(2), h(k-1)\}$. We then then further use $h(2)$ and $h(k-1)$ to derive an upper bound on $g(l)$. 

Indeed, by the following direct calculation, we can easily prove this:

Note that its derivative with respect to $\ell$ is $$h'(\ell) = -\frac{\log{\ell} + \log{(n-k) - 2\log{k}}}{\log{n}} + \frac{2\ell-1}{2k}.$$ Further calculate its second derivative: $$h''(\ell) = -\frac{1}{\log{n}}\frac{1}{\ell} + \frac{1}{k}$$. Note that $\ell_0 = k / \log{n}$ is the only solution of $h''(\ell)=0$. Easy to check that $\ell_0 \leq k-1$. Therefore, we have the following two cases: 
\begin{description}\denselist
\item[Case (\romannumeral 1)] If $\ell_0 < 2$, then $h'(\ell)$ is strictly increasing on $\ell\in[2,k-1]$;
\item[Case (\romannumeral 2)] If $\ell_0 \in [2, k-1]$, then $h'(\ell)$ is strictly decreasing on $[2, \ell_0]$ and strictly increasing on $[\ell_0, k-1]$.
\end{description}
Note that 
$$h'(2) < -\frac{\log{2}+\log{(n/2)} - 2\log{k}}{\log{n}} + \frac{3}{2k} = -1 +\frac{2\log{k}}{\log{n}} + \frac{3}{2k} < -1 + \frac{1}{2} + \frac{3}{20} < 0.$$ 
Thus in either case $h'(\ell)$ can become $0$ at most once within $\ell \in [2, k-1]$, and we have $\max_{\ell\in [2,k-1]}h(\ell) = \max \{h(2), h(k-1)\}$. 

Routine calculation shows that (using that $n-k > n/2$), for $n$ large enough: 
\begin{align*}
h(2) &< - \frac{2\left[\log{(2(n /2)) - 1 - 2\log k}\right]}{\log n} + \frac{1}{k} = -2 + \frac{1}{k} + \frac{2}{\log n}  + \frac{4\log k}{\log n} < -\frac{1}{2},\\
h(k-1) &< -\frac{(k-1)\left[\log{(n / 2)} -1 -\log k - \log \left(k /(k-1)\right)\right]}{\log n} + \frac{k^2-3k +2}{2k} \\
&< \left[\frac{k^2-3k +2}{2k} - (k-1)\right] + \frac{k(1+ \log 2)}{\log n} + \frac{k\log k }{\log n} + \frac{k\log \left(k / (k-1)\right)}{\log n} \\
&< -\frac{1}{2} + \frac{1}{10} - \frac{k}{6} < -\frac{1}{2}.
\end{align*}
Thus, $\forall \ell \in [2, k-1]$, we have $g(\ell) < n^{-1 /2}$.
It then follows that 
$$\sum\limits_{\ell=2}^{k-1}g(\ell)  < k \cdot n^{-\frac{1}{2}} \leq n^{\frac{1}{4}} \cdot n^{-\frac{1}{2}} = n^{-\frac{1}{4}}.$$ 
Hence by Eqn (\ref{eqn:2nd_delta*/mu}), we have $\Delta^* / \myE[Z] < n^{-1 /4} \rightarrow 0$, and therefore $\myProb[Z = 0] \rightarrow 0$ by Theorem \ref{thm:2ndmoment}.
\end{proof}

%%%%%%%%%%%%%%%%%%%%%%%%%%%%%%%%%%%%%%%%%%%%%%%%%%%%%%%%%%%%%%%%%%%%%%%%%%%
\subsection{Proof of Theorem \ref{thm:type_2_quite_sparse}}
\label{appendix:thm:type_2_quite_sparse}

%\begin{proof}[Proof of Theorem \ref{thm:type_2_quite_sparse}.]
\paragraph{\bf Proof of part (a).~} We use the same notation $\tilde{A}_{uv}$ and $B_{uv}$ as in the proof of Theorem \ref{thm:type_2_very_sparse}. Now we set 
\begin{align*}
N_{\max} := \frac{5\log{n}}{\log{ \left(\log{n} / (\sigma \theta 2^d nr^d)\right)}}.
\end{align*}
Again, denote $\aF$ to be the event that ``for every $v \in \mathcal{X}_n$, the ball $B_{r}(v)\cap \mathcal{X}_n$ contains at most $N_{\max}$ points''; and $\aF^c$ denotes the complement event of $\aF$. By Lemma \ref{lem:num_points_1rball_quite_sparse}, we know that, $\myprob[\aF^c] = O(n^{-3})$.

Let $\aK_n$ be a positive number to be determined such that $\aK_n \rightarrow \infty$ as $n \rightarrow \infty$. By applying the pigeonhole principle and the union bound, we have: 
\begin{align}\label{eqn:mainresult_quite_sparse}
&\myprob\left[ \omega_{u,v}\left(G_n^{0,\qqq}\right) \ge \aK_n  \middle| \aF \right] \nonumber \\
\le  &\myprob\left[ \omega_{u,v}\left(G_n^{0,\qqq}|_{\tilde{A}_{uv}}\right) \ge \aK_n / 2\middle| \aF\right] + \myprob\left[ \omega_{u,v}\left(G_n^{0,\qqq}|_{B_{uv}}\right)\ge \aK_n / 2 \middle| \aF\right] 
\end{align}

\paragraph{\bf Case (\romannumeral 1): bounding the first term in Eqn. (\ref{eqn:mainresult_quite_sparse}).~} 
Applying Theorem \ref{lem:BCLdoubling} for points in $A_{uv}$ gives a \myWSP{} $\mathcal{P} = \{P_i\}_{i\in \Lambda}$ of $A_{uv}$ with $|\Lambda| \le \myBC^2$ being a constant. Augment each $P_i$ to $\tilde{P}_i = P_i \cup \{u\} \cup \{v\}$. 
Check Figure \ref{fig:twocases_very_sparse} (a). Again, by applying pigeonhole principle and the union bound, we have: 
\begin{align}\label{eqn:caseAunionbound_quite_sparse}
&\myprob\left[ \omega_{u,v}\left(G_n^{0,\qqq}|_{\tilde{A}_{uv}}\right) \ge \aK_n / 2 \middle| \aF\right] 
\leq \sum_{i=1}^{|\Lambda|} \myprob\left[\omega_{u,v}\left(G_n^{0,\qqq}|_{\tilde{P}_{i}}\right) \ge \aK_n / (2|\Lambda|) \middle| \aF\right]
\end{align}    

Now set $k_n := \floor*{\aK_n / (2|\Lambda|)}-2$. Easy to see that $k_n \rightarrow \infty$ as $n \rightarrow \infty$. Same as in the proof for Theorem \ref{thm:type_2_very_sparse}, we have
\begin{align}\label{eqn:expectationX_quite_sparse}
&\myprob \left[\omega_{u,v}\left(G_n^{0,\qqq}|_{\tilde{P}_{i}}\right) 
 \ge  k_n+2 \middle| \aF\right] \nonumber\\
 \leq & ~\qqq^{2k_n}\sum\limits_{\substack{x_1 + x_2 + \cdots + x_m = k_n \\ 0 \leq x_i \leq N_{\max}}} \binom{N_{\max}}{x_1} \binom{N_{\max}}{x_2}\cdots \binom{N_{\max}}{x_m}\qqq^{(k_n^2-\sum_{i=1}^{m}x_i^2)/2} 
\end{align}
where $m \leq n$ is the number of $C_i^{(j)}$ in the clique-partition $\tilde{P}_i$.

If $\aK_n \leq 2N_{\max}$, then $k_n \in [1, N_{\max}]$. By applying Lemma \ref{lem:insertioncaseA_very_sparse}, we have that if $1 \leq k_n \leq N_{\max}$, then there exist constants $c^a_1$ and $c^a_2$ (which depend on the \Bconst{} $\myBC$), such that if 
\begin{align*}
\qqq \le c^a_1 \cdot \left(\frac{1}{n}\right)^{c^a_2/\aK_n} \frac{\aK_n}{N_{\max}},
\end{align*}
then the right hand side of Eqn. (\ref{eqn:expectationX_quite_sparse}) is $O(n^{-3})$. 

Thus, following the same argument in the proof of part (a) of Theorem \ref{thm:type_2_very_sparse}, we have
\begin{align}\label{eqn:caseAalmostfinal_quite_sparse}
\text{If}~ \aK_n \leq 2N_{\max} &~\text{and}~ \qqq \leq c^a_1 \cdot \left(1 /n\right)^{c^a_2/\aK_n} (\aK_n / N_{\max}),\nonumber\\
 &~\text{then}~\myprob\left[ \omega_{u,v}\left(G_n^{0,\qqq}|_{\tilde{A}_{uv}}\right) \ge \aK_n / 2 \middle| \aF\right] = O(n^{-3}). 
\end{align}
Finally, suppose $\aK_n > \aK_0 = 2N_{\max}$. 
Using Eqn (\ref{eqn:caseAalmostfinal_quite_sparse}), we know that if 
\begin{align*}
\qqq \le c^a_1 \cdot \left(\frac{1}{n}\right)^{c^a_2/\aK_0}\frac{\aK_{0}}{N_{\max}} = 2c_1^a \left(\frac{\sigma \theta 2^d \left(nr^d\right)}{\log n} \right)^{\frac{c_2^a}{10}} = \left(2c_1^a \cdot \left(\sigma \theta 2^d \right)^{\frac{c_2^a}{10}}\right) \left(\frac{nr^d}{\log n} \right)^{\frac{c_2^a}{10}}
\end{align*}
and $\aK_n > \aK_0$, then $$\myprob\left[ \omega_{u,v}\left(G_n^{0,\qqq}|_{\tilde{A}_{uv}}\right) \ge \aK_n / 2\middle| \aF\right] \le \myprob\left[ \omega_{u,v}\left(G_n^{0,\qqq}|_{\tilde{A}_{uv}}\right) \ge \aK_0 / 2 \middle| \aF\right] = O(n^{-3}). $$
% Then we have 
% $$\myprob\left[ G_n^{0,q}|_{\tilde{A}_{uv}}~\text {has a}~uv\text{-clique of size} \ge \frac{\aK_n}{2}\middle| \aF\right] \le \myprob\left[ G_n^{0,q}|_{\tilde{A}_{uv}}~\text {has a}~uv\text{-clique of size} \ge \frac{\aK_0}{2}\middle| \aF\right] = O(n^{-3}),$$
% where the second bound follows from Eqn (\ref{eqn:caseAalmostfinal_quite_sparse}) and assume that $q$ is
% \begin{align*}
% q \le c^a_1 \cdot \left(\frac{1}{n}\right)^{c^a_2/\aK_0}\frac{\aK_{0}}{N_{\max}} = 2c_1^a \left(\frac{\sigma \theta 2^d \left(nr^d\right)}{\ln n} \right)^{\frac{c_2^a}{10}} = \left(2c_1^a \cdot \left(\sigma \theta 2^d \right)^{\frac{c_2^a}{10}}\right) \left(\frac{nr^d}{\ln n} \right)^{\frac{c_2^a}{10}}. 
% \end{align*}
Set $C_1^a := 2c_1^a \cdot \left(\sigma \theta 2^d \right)^{c_2^a/10}$ and $C_2^a := c_2^a / 10$ be two constants. Combining this with Eqn. (\ref{eqn:caseAalmostfinal_quite_sparse}), we thus obtain that: 
\begin{align}\label{eqn:caseAfinal_quite_sparse}
\text{If}~ \aK_n \rightarrow \infty &~\text{and} ~\qqq \leq \min\left\{C^a_1 \cdot \left(nr^d / \log n\right)^{C^a_2} , ~c^a_1 \cdot \left(1 / n\right)^{c^a_2/\aK_n} (\aK_n / N_{\max})\right\}, \nonumber\\
&~\text{then}~\myprob\left[ \omega_{u,v}\left(G_n^{0,\qqq}|_{\tilde{A}_{uv}}\right) \ge \aK_n / 2\middle| \aF\right] = O(n^{-3}). 
\end{align}

\noindent{\bf Case (\romannumeral 2): bounding the second term in Eqn. (\ref{eqn:mainresult_quite_sparse}).} 
Recall that $B_{uv} = B_r^{\mathcal{X}_n}(u) \cup B^{\mathcal{X}_n}_r(v)$ (see Figure \ref{fig:twocases_very_sparse} (b)). We again use the notation $\tilde{G}_{uv}^{local}$ defined in the proof of part (a) of Theorem \ref{thm:type_2_very_sparse}. Set $N_u := |B^{\mathcal{X}_n}_r(u)|$ and $N_v := |B^{\mathcal{X}_n}_r(v)|$. Let $\tilde{k}_n := \floor*{\aK_n / 2} - 2$. Easy to see $\tilde{k}_n \geq 1$. Using the same argument as in Case (\romannumeral 2) in the proof of Theorem \ref{thm:type_2_very_sparse}, we have
\begin{align}\label{eqn:expectationI_quite_sparse}
\myprob\left[ \omega_{u,v}\left(\tilde{G}_{uv}^{local}\right) \ge \aK_n / 2 \middle| \aF \right] \leq \sum\limits_{\substack{x_1 + x_2 = \tilde{k}_n \\ 0 \leq x_1 \leq N_u - 1\\ 0\le x_2 \le N_v - 1}} \binom{N_u-1}{x_1} \binom{N_v-1}{x_2}\qqq^{(x_1+1)(x_2+1)-1}
\end{align}

By applying Lemma \ref{lem:insertioncaseB_very_sparse}, we know that there exist constants $c^b_1$ and $c^b_2$ (which depend on the \Bconst{} $\myBC$), such that if $\aK_n \rightarrow \infty$ and
\begin{align*}
\qqq \le c^b_1 \cdot \left(\frac{1}{n}\right)^{c^b_2/\aK_n} \frac{\aK_n}{N_{\max}},
\end{align*}
then the right hand side of Eqn. (\ref{eqn:expectationI_quite_sparse}) is $O(n^{-3})$. 
That is,
\begin{align}\label{eqn:caseBfinal_quite_sparse}
\text{if} ~\aK_n \rightarrow \infty &~\text{and}~ \qqq \leq ~c^b_1 \cdot \left(1 /n\right)^{c^b_2/\aK_n} (\aK_n / N_{\max}) ,\nonumber\\ &~\text{then} ~ \myprob \left[\omega_{u,v}\left(\tilde{G}_{uv}^{local}\right) \ge \aK_n / 2 \middle| \aF \right] = O(n^{-3})
\end{align}

Pick 
\begin{align*}
\aK_n = 4N_{\max} = \frac{20\log{n}}{\log{\left(\log{n} / (\sigma \theta 2^d nr^d) \right)}} = \frac{20\log{n}}{\log{\left(\log{n} / nr^d \right)}} + \text{const.}. 
\end{align*}
%\yusu{Minghao: why not change the last term to simply $\Theta(\frac{\ln{n}}{\ln{\frac{\ln{n}}{nr^d}}})$. The current one looks weird, and not very precise/correct unless we use $\pm$ instead of $+$.} \minghao{In my humble opinion, it's not weird and is crucial. Const. can also be negative.}
Note that this makes the first term of the constraint on $\qqq$ in Eqn. (\ref{eqn:caseAfinal_quite_sparse}) dominate. %\yusu{Minghao, it also makes the first term in Eqn (\ref{eqn:caseAfinal_quite_sparse}) domintes, no? So add it that there too.} \minghao{Fixed.}
Thus, combining Eqn. (\ref{eqn:caseBfinal_quite_sparse}), (\ref{eqn:caseAfinal_quite_sparse}) and (\ref{eqn:mainresult_quite_sparse}), there exist constants $C_1 = \min\{C_1^a, c_1^b\}$ and $C_2 = \max\{C^{a}_2, c^{b}_2/10\}$ such that if $\qqq$ satisfies conditions in Eqn. (\ref{eqn:type_2_qbound_quite_sparse}), then 
\begin{align*}
\myprob\left[\omega_{u,v}\left(G_n^{0,\qqq}\right) \ge \aK_n\right] \leq \myprob\left[\omega_{u,v}\left(G_n^{0,\qqq}\right) \ge \aK_n \middle| \aF\right] + \myprob[\aF^c]  = O(n^{-3})
\end{align*}

Finally, by applying the union bound, we derive that with high probability, for each of the $O(n^2)$ long-edge $(u,v)$, its edge clique number
\begin{align*}
\omega_{u,v}(G_n^{0,\qqq}) \lesssim \frac{\log{n}}{\log{\left(\log{n} / nr^d\right)}}
\end{align*}
as long as Eqn. (\ref{eqn:type_2_qbound_quite_sparse}) holds. This completes the proof of Part (a) if Theorem \ref{thm:type_2_quite_sparse}.

%\yusu{Minghao: btw, the macro ``myparagraph'' was what I used before to save space. but for this journal version, we don't need that. Please change all of them to the standard ``paragraph" macro, because this standard one does better adjustments around other environments (e.g, figures/equations). I already changed a bunch whenever I see them. }

\myparagraph{\bf Proof of part (b).~} We again try to bound the two terms on the right hand side of Eqn. (\ref{eqn:mainresult_quite_sparse}) from above separately. For case (\romannumeral 1), our result relies on the following lemma.

\begin{lemma}\label{lem:insertioncaseA_quite_sparse_looser}
There exists a constant $C_3 > 0$ depending on the \Bconst{} $\myBC$ such that if $\left(nr^d / \log n \right)^{4 / 15} \leq \qqq < 1$ and $\aK_n = C_3 \floor*{\log_{1 / \qqq}n}$, then we have
\begin{align*}
\myprob\left[ \omega_{u,v}\left(G_n^{0,\qqq}|_{\tilde{A}_{uv}}\right) \ge \aK_n / 2\middle| \aF\right] = O(n^{-3}).
\end{align*}
\end{lemma}

\begin{proof}
By a similar argument in Appendix \ref{appendix:lem:insertioncaseA_very_sparse_looser}, we know that 
\begin{align}
\myprob\left[ \omega_{u,v}\left(G_n^{0,\qqq}|_{\tilde{A}_{uv}}\right) \ge \aK_n / 2\middle| \aF\right] \leq  \frac{1}{\sqrt{2\pi}} \left(\frac{emN_{\max}}{k_n}\qqq^{\frac{1}{2}k_n - \frac{3}{2}N_{\max} + 1} \right)^{k_n}\label{eqn:mainestimate_quite_sparse_looser}
\end{align}
where $k_n = \floor*{\aK_n / (2|\Lambda|)}-2$. Pick a constant $C_3$, which only depends on the \Bconst{} $\myBC$, such that $$\frac{C_3 \floor*{\log_{1 /\qqq}n}}{2 |\Lambda|} - 3 \geq 16 \log_{1 /\qqq}n \geq 6 N_{\max}.$$ This can be done since we have $\log n / nr^d \rightarrow \infty$ and $\left(nr^d / \log n \right)^{4 /15} \le \qqq  < 1$ and thus
\begin{align*}
\log_{1 /\qqq}n \geq \frac{15 \log n}{4 \log \left(\log n / nr^d\right)} &= \frac{3}{8} \frac{5 \log n}{ (1/2)\log \left(\log n / nr^d\right)} \\
&> \frac{3}{8} \frac{5\log n}{\log \left(\log n / (\sigma \theta 2^d nr^d) \right)} = \frac{3}{8} N_{\max}.
\end{align*}

Set $\aK_n = C_3 \floor*{\log_{1 / \qqq}n}$. Recall that $k_n = \floor*{\aK / (2 |\Lambda|)} - 2$, thus $k_n \geq 16 \log_{1 / \qqq}n \geq 6 N_{\max}$. Finally, we use Eqn. (\ref{eqn:mainestimate_very_sparse_looser_continue}) (with $k$ being replaced by $k_n$) to complete the proof. 
\end{proof}

Now pick such $C_3$ in Lemma \ref{lem:insertioncaseA_quite_sparse_looser}. We know that the following statement holds.
\begin{align}\label{eqn:caseAalmostfinal_quite_sparse_looser}
\text{If}~ &\left(nr^d / \log n \right)^{4 / 15} \leq \qqq < 1, \nonumber\\ &~~\text{then}~\myprob\left[ \omega_{u,v}\left(G_n^{0,\qqq}|_{\tilde{A}_{uv}}\right) \ge C_3\floor*{\log_{1 /\qqq}n} / 2\middle| \aF\right] = O(n^{-3}). 
\end{align}

For case (\romannumeral 2), we know that if event $\aF$ has already happened, then $|B_{uv}| \leq 2N_{\max}$. Note that if $\left(nr^d / \log n \right)^{C_3 / 42} \leq \qqq < 1$, then we have
\begin{align*}
\frac{C_3\floor*{\log_{1 /\qqq}n}}{2} &> 2\frac{5 \log n}{(1/2)\log \left(\log n / nr^d \right)} + \left(\frac{\log n}{\log \left(\log n / nr^d \right)} -\frac{C_3}{2}\right) \\
&> 2\frac{5\log{n}}{\log{\left(\log{n} / (\sigma \theta 2^d nr^d)\right)}} = 2N_{\max} \geq |B_{uv}|.
\end{align*}
Set $\xi = \min\left\{4 / 15, C_3 / 42 \right\}$. Hence, we obtain that:
\begin{align}\label{eqn:caseBalmostfinal_quite_sparse_looser}
\text{If}~ \left(nr^d / \log n \right)^{\xi} \leq \qqq < 1, ~~\text{then}~\myprob\left[ \omega_{u,v}\left(G_n^{0,\qqq}|_{B_{uv}}\right) \ge C_3\floor*{\log_{1 /\qqq}n} / 2\middle| \aF\right] = 0. 
\end{align}
Thus, combining Eqn. (\ref{eqn:caseBalmostfinal_quite_sparse_looser}), (\ref{eqn:caseAalmostfinal_quite_sparse_looser}) and (\ref{eqn:mainresult_quite_sparse}), we know that if $\left(nr^d / \log n \right)^{\xi} \leq \qqq < 1$, then with high probability, for every long-edge $(u,v)$, its edge clique number $\omega_{u,v}\left( G_n^{0,\qqq}\right) \lesssim \log_{1 / \qqq}n$. This completes the proof of Theorem \ref{thm:type_2_quite_sparse}.
%\end{proof}

%%%%%%%%%%%%%%%%%%%%%%%%%%%%%%%%%%%%%%%%%%%%%%%%%%%%%%%%%%%

%%%%%%%%%%%%%%%%%%%%%%%%%%%%%%%%%%%%%%%%%%%%%%%%%%%%%%%%%%%%%%%%%%%%%%%%%%%%%%%%%%%%%%%

\section{The missing proofs in Section \ref{sec:proof_deletion_only_RGG}}
\label{appendix:thm:deletion_only_dense}

\subsection{Proof of Part (\rom{3}) --- supercritical regime}\label{appendix:sec:deletion_only_dense}
In this section, we discuss the order of $\omega(G_n^{\pp,0})$ in the regime $\sigma n r^d / \log n \rightarrow t \in (0, \infty)$. 

\myparagraph{\bf Proof of upper bound.~} We first focus on the upper bound of $\omega\left(G_n^{\pp,0}\right)$. Let $N$ be a random variable sampled from $Poisson\left(\left(1+\delta\right)n\right)$ for some $\delta > 0$ (say $\delta = 1/2$). Note that $G_N$ is an $r-$neighborhood graph of the Poisson point process $\mathcal{P}_{(1+\delta)n}$ with intensity $(1+\delta)nf$. Completely analogously to the proof of upper bound in Section \ref{sec:deletion_only_quite_sparse}, we have
\begin{align*}
\myprob \left[ \omega\left(G_N^{\pp,0} \right) \geq k_n \right] & < \frac{1}{\sqrt{2\pi}} \left(\frac{(1+\delta)e\sigma\theta nr^d (1-\pp)^{(k_n-1)/2}}{k_n} \right)^{k_n}
\end{align*}
Finally, pick $k_n = 3\log_{1 / (1-\pp)}nr^d$. Since in the supercritical regime $\sigma n r^d / \log n \rightarrow t \in (0, \infty)$, routine calculations show that $\myprob \left[ \omega\left(G_N^{\pp,0} \right) \geq k_n \right] = o(1)$. Hence, a.s.
\begin{align*}
\omega\left(G_n^{\pp,0}\right) \lesssim \log\left(nr^d\right) 
\end{align*}

%\yusu{Minghao, is the fact that $\frac{\sigma n r^d}{\ln n} \rightarrow t \in (0, \infty)$ used anywhere here? } 

\paragraph{\bf Proof of lower bound.~} Now let us move on to the lower bound of $\omega\left(G_n^{\pp,0} \right)$. For this regime, we need slightly stronger condition on the range of $t$. That is, $\sigma nr^d \geq T \log n$ for some constant $T >0$ to be determined. 

Completely analogously to the proof of lower bound in Section \ref{sec:deletion_only_quite_sparse}, let $N$ be a random variable sampled from $Poisson\left((1-\delta')n\right)$ for some $\delta' \in (0,1)$ (say $\delta' = 1 / 2$). Set $k_n$ be an integer to be determined. Now, we have
\begin{align}
\myprob \left[ \omega\left(G_n^{\pp,0} \right) \leq k_n \right] \leq \myprob \left[ \omega\left(G_N^{\pp,0} \right) \leq k_n \right] + e^{-\gamma' n}\nonumber
\end{align}
for some constant $\gamma' > 0$ (depending on $\delta'$) by Lemma \ref{lem:chernoff_hoeffding}. Now fix some constant $\rho \in (0,1)$ (say $\rho = 1 /2$). Recall $W_{1/2} = B_{1/2}(0)$. By Lemma \ref{lem:deletion_partition_lemma}, there exist points $x_1, x_2, \cdots, x_m$ with $m = \Omega\left(r^{-d}\right) \geq 1$ such that the sets $x_i + W_{1/2}$ are disjoint and $$\nu \left(x_i + W_{1/2} \right) \geq \frac{(1-\rho) \sigma \theta}{2^d} r^d$$ for $i = 1, \cdots, m$. Let $\mathbf{X}_i$ be the set of points of $G_N$ falling in $x_i + W_{1/2}$. Then, we have
\begin{align*}
\myprob \left[ \omega\left(G_N^{\pp,0} \right) \leq k_n \right] &\leq \myprob \left[ \omega\left(G_N^{\pp,0} \mid_{\mathbf{X}_1} \right) \leq k_n, \cdots, \omega\left(G_N^{\pp,0} \mid_{\mathbf{X}_m} \right) \leq k_n \right]\\
& = \prod\limits_{i=1}^{m} \myprob \left[ \omega\left(G_N^{\pp,0} \mid_{\mathbf{X}_i} \right) \leq k_n\right]
\end{align*}
Easy to see that all the points falling in any $r/2$-ball span a clique in $G_N$. Thus, for each $i$, we have 
\begin{align*}
\myprob \left[ \omega\left(G_N^{\pp,0} \mid_{\mathbf{X}_i} \right) \leq k_n\right] = \myprob \left[ \omega\left(G\left(|\mathbf{X}_i|, 1-\pp\right) \right) \leq k_n\right].
\end{align*}
Set 
\begin{align*}
\Phi_n := \frac{(1-\rho)(1-\delta')\sigma \theta nr^d}{2^{d+1}}
\end{align*}
which goes to infinity as $n$ grows. Note that $|\mathbf{X}_i| \sim Poisson(\tilde{\lambda})$ where $\tilde{\lambda} := (1-\delta')n \cdot \nu(x_i + W_{1/2}) \geq 2 \Phi_n$ (see Eqn. (\ref{eqn:lambdabound})). %\yusu{Minghao: I added a ref to Eqn (\ref{eqn:lambdabound}). Also, should we change $\lambda$ to $\tilde{\lambda}$ there to be consistent with the main text too? } 
Now pick $k_n := \floor*{\log_{1 / (1-\pp)}\Phi_n} = \Omega\left( \log (nr^d)\right)$. By the law of total probability, we have
\begin{align}
&\myprob \left[ \omega\left(G\left(|\mathbf{X}_i|, 1-\pp\right) \right) \leq k_n\right] \nonumber\\
\leq& \myprob\left[|\mathbf{X}_i| \leq \Phi_n \right] + \sum\limits_{j = \ceil*{\Phi}}^{\infty}\myprob \left[ \omega\left(G\left(j, 1-\pp\right) \right) \leq k_n\right] \myprob\left[|\mathbf{X}_i| = j\right]\nonumber\\
 \leq& \myprob\left[|\mathbf{X}_i| \leq \frac{\tilde{\lambda}}{2} \right] + \sum\limits_{j = \ceil*{\Phi}}^{\infty}\myprob \left[ \omega\left(G\left(j, 1-\pp\right) \right) \leq \floor*{\log_{\frac{1}{1-\pp}}j}\right] \frac{e^{-\tilde{\lambda}} \tilde{\lambda}^j}{j!}\nonumber\\
 <& e^{-\frac{1}{10}\tilde{\lambda}} + \sum\limits_{j = \ceil*{\Phi}}^{\infty} e^{-j} \frac{e^{-\tilde{\lambda}} \tilde{\lambda}^j}{j!}\label{eqn:deletion_only_lower_bound_law_of_total}\\
<& e^{-\frac{1}{10}\tilde{\lambda}} + e^{-\tilde{\lambda}}\sum\limits_{j = 0}^{\infty} \frac{ \left(\tilde{\lambda} / e\right)^j}{j!} \nonumber\\
=& e^{-\frac{1}{10}\tilde{\lambda}} + e^{-\tilde{\lambda}} \cdot e^{\tilde{\lambda} /e} \nonumber\\
<& 2e^{-\frac{1}{10}\tilde{\lambda}}\nonumber  
\end{align}
Inequality (\ref{eqn:deletion_only_lower_bound_law_of_total}) holds due to Lemma \ref{lem:chernoff_hoeffding} and Lemma \ref{lem:ER_lowerbound_well-known}. Now set
\begin{align*}
T := \frac{10 \cdot 2^d}{(1-\rho)(1-\delta') \theta}.
\end{align*}
Note that $\sigma nr^d \geq T \log n$. Then 
\begin{align*}
e^{-\frac{1}{10}\tilde{\lambda}} \leq e^{-\frac{(1-\rho)(1-\delta')\theta}{10 \cdot 2^d}(\sigma nr^d)} \leq e^{-\log n} = n^{-1}
\end{align*}
It follows that $\myprob \left[ \omega\left(G_n^{\pp,0} \right) \leq k_n \right] = o(1)$ with $k_n = \Omega\left( \log (nr^d)\right)$, which concludes the proof of part (\rom{3}) of Theorem \ref{thm:deletion_only_RGG}.

%%%%%%%%%%%%%%%%%%%%%%%%%%%%%%%%%%%%%%%%%%%%%%%%%%%%%%%%%%%%%%%%%%%%%%%%%%%%%%%%%%%%%%%%%%%%%%%%%%%%%%%%%%%%%%%%%%%%%%%%%%%%%%%%%%

\end{document}